\documentclass{amsart}

\usepackage[mathscr]{eucal}
\usepackage{verbatim,url} 
\usepackage{amsthm,amssymb,latexsym,bm,mathtools}
\usepackage{tikz}
\usepackage{a4wide}

\usetikzlibrary{arrows,
positioning,calc, 
decorations.pathmorphing,fit,matrix,decorations}

\newtheorem{thm}{Theorem}[section]
\newtheorem{lem}[thm]{Lemma}
\newtheorem{coro}[thm]{Corollary}
\newtheorem{prop}[thm]{Proposition}

\theoremstyle{definition}

\newenvironment{newlist}
   {\begin{list}{}{\setlength{\labelsep}{0.25cm}
                   \setlength{\labelwidth}{0.65cm}
                      \setlength{\leftmargin}{0.9cm}}}
   {\end{list}}

\newenvironment{longnewlist}
   {\begin{list}{}{\setlength{\labelsep}{0.25cm}
                   \setlength{\labelwidth}{1.35cm}
                      \setlength{\leftmargin}{1.5cm}}}
   {\end{list}}

\newcommand{\class}[1]{\boldsymbol{\mathfrak{#1}}}
\newcommand{\cat}[1]{\boldsymbol{\mathscr{#1}}}
\newcommand{\str}[1]{\mathbf{#1}}
\newcommand{\alg}[1]{\str{#1}}
\newcommand{\spc}[1]{\str{#1}}
\newcommand{\fnt}[1]{\mathsf{#1}}
\newcommand{\cnst}[1]{\boldsymbol{#1}}
\newcommand{\ope}[1]{\mathbb{#1}}
\newcommand{\defn}[1]{{\emph{#1}}}

\newcommand{\B}{\alg{B}}
\newcommand{\CA}{{\cat A}}
\newcommand{\CX}{\cat X} 
\newcommand{\Z}{\cat Z} 
\newcommand{\CP}{\cat P} 
\newcommand{\CY}{\cat Y}  

\newcommand{\CCD}{\cat D}
\newcommand{\DB}{\cat{DB}}

\newcommand{\DPB}{\cat{DPB}}
\newcommand{\unbounded}[1]{#1\mbox{\tiny{$u$}}}

\newcommand{\DBU}{\unbounded{\DB}}
\newcommand{\DPBU}{\unbounded{\DPB}}
\newcommand{\DU}{\unbounded{\CCD}}

\newcommand{\twiddle}[1]{{\smash{\underset{\raise.375ex\hbox{$\smash\sim$}}
       {\mathbf{#1}}}\vphantom{\underline{\str{#1}}}}}

\newcommand{\D}{\fnt D}
\newcommand{\E}{\fnt E}
\newcommand{\U}{\fnt U}
  
\newcommand{\Uu}{\unbounded{\U}}
\newcommand{\Hu}{\unbounded{\fnt H}}
\newcommand{\Ku}{\unbounded{\fnt K}}
\newcommand{\Du}{\unbounded{\D}}
\newcommand{\Eu}{\unbounded{\E}}

\newcommand{\KH}{\fnt{KH}}
\newcommand{\ED}{\fnt{ED}}
\newcommand{\DE}{\fnt{DE}}

\newcommand{\A}{\alg{A}}
\newcommand{\M}{\alg{M}}
\newcommand{\CM}{\class{M}}
\newcommand{\two}{\boldsymbol 2}
\newcommand{\twoU}{\unbounded{\two}} 
\newcommand{\Lalg}{\alg{L}}
\newcommand{\fourDB}{\boldsymbol 4}
\newcommand{\four}{\boldsymbol 4}
\newcommand{\fourU}{\unbounded{\four}}

\newcommand{\fourDBT}{\twiddle 4}
\newcommand{\fourDBTU}{\twiddle{\unbounded{4}}}

\newcommand{\MT}{\twiddle{\spc{M}}}
\newcommand{\CMT}{\twiddle{\CM}}
\newcommand{\twoT}{\twiddle 2}
\newcommand{\X}{\spc{X}}
\newcommand{\Y}{\spc{Y}}

\newcommand{\Tp}{{\mathscr{T}}}

\newcommand{\w}{\omega}
\newcommand{\zerobar}{\overline{\boldsymbol 0}}
\newcommand{\onebar}{\overline{\boldsymbol 1}}

\DeclareMathOperator{\ISP}{\ope{ISP}}
 
 \DeclareMathOperator{\Su}{\ope{S}}
\DeclareMathOperator{\IScP}{{\ope{IS} _{\mathrm{c}}
\ope{P}^+}}

\renewcommand{\leq}{\leqslant}
\renewcommand{\geq}{\geqslant}
\newcommand{\du}{\smash{\,\cup\kern-0.45em\raisebox{1ex}{$\cdot$}}\,\,}

\renewcommand{\bar}{\overline}

\renewcommand{\emptyset}{\varnothing}

\begin{document}

\title[]
{Distributive bilattices from the perspective of natural duality theory}


\author{L. M. Cabrer}
\email{lmcabrer@yahoo.com.ar}
\address{Mathematical Institute \\
University of Oxford\\
Radcliffe Observatory Quarter, Oxford OX2 6GG\\ UK}

\author{H. A. Priestley}
\email{hap@maths.ox.ac.uk}
\address{Mathematical Institute \\
University of Oxford\\
Radcliffe Observatory Quarter, Oxford OX2 6GG\\ UK}

\begin{abstract}  This paper provides a fresh perspective on the representation of distributive bilattices and of related varieties.  The techniques of natural
duality are employed to give, economically and in a uniform way, categories of
structures  dually equivalent to these varieties.
We relate our dualities to the product representations for bilattices and to pre-existing dual representations by a 
 simple 
translation process which is an instance of a more general 
mechanism for 
connecting   
 dualities based on Priestley duality to natural dualities.
Our approach gives us access to descriptions of algebraic/categorical properties of bilattices and also reveals how `truth' and `knowledge' may be seen as dual notions.
\end{abstract}

\maketitle

\section{Introduction} \label{sec:intro}

 This paper is the first of  three devoted to bilattices,
 the other two being 
\cite{CCP,CP2}.  
Taken together, our three papers  
  provide a  systematic treatment of
 dual representations via natural duality theory, showing  
that this theory  
applies in a uniform way to 
 a range of varieties 
having bilattice structure as a  unifying theme.  
The representations are based on hom-functors and hence the constructions are inherently functorial. 
The key theorems on which we call  
 are easy to apply,  in black-box fashion,  without the need to delve  into the  theory. 
Almost all of the natural duality theory  we employ 
can, if 
desired, 
be found  in the text by Clark and Davey~\cite{CD98}. 

The term bilattice, loosely,
refers to a set  $L$ equipped with   
two lattice orders, $\leq_t$ and $\leq_k$, subject to 
some compatibility requirement. 
The subscripts have the following connotations:  
$t$ measuring `degree of truth' 
and $k$  `degree of
knowledge'. 
As an algebraic structure, then, a bilattice carries two pairs 
of lattice operations:  $\land_t$ and $\lor_t$; $\land_k$ and $\lor_k$. The term distributive is  applied when all possible distributive laws hold amongst these four operations;  
distributivity imposes  strictly stronger compatibility between the two lattice structures  than the condition 
known as interlacing.
Distributive bilattices may be, but need not be, 
also assumed to have universal bounds for 
each order which are  treated as distinguished constants
(or, in algebraic terms, as nullary operations). 
  In addition,  a bilattice is  often, but not always, assumed  
to carry in addition an involutory unary operation $\neg$, thought of as modelling a negation. 
Historically, the investigation of bilattices (of all types) has been   tightly bound up with their potential role as 
models in artificial intelligence and with the study of associated 
logics. We note, by way of a sample,  the pioneering papers of Ginsberg \cite{Gin} and Belnap
\cite{ND1,ND2} 
and the more recent works 
\cite{AA1,RPhD,BR11}.
 We do not, except to a very limited extent in Section~\ref{Sec:Conclude}, 
address
logical aspects of bilattices in our work.

In 
this  
 paper we focus on   distributive bilattices,
with or  without bounds  and with or without negation.
In \cite{CP2} we consider varieties arising as expansions of 
those considered here, in particular distributive bilattices with both
negation and a  conflation operation. 
In \cite{CCP} we move outside  the realm of distributivity, and even 
outside the wider realm of interlaced bilattices, and study 
 certain quasivarieties generated by finite non-interlaced bilattices
arising in connection with default logics. 

The present paper is organised as follows.  Section~\ref{sec:DBilat}
formally introduces the varieties we shall study and establishes some basic properties.  Sections~\ref{sec:DB},~\ref{sec:DBU} and \ref{sec:DPB}
present our natural dualities for these varieties.  We preface these sections  by accounts of the relevant natural duality theory, 
tailored to 
our intended applications (Sections~\ref{piggyonesorted} and~\ref{sec:multi}).
Theory and  practice 
are  brought together in Sections~\ref{Sec:DisPiggyDual}
and~\ref{sec:prodrep}, in which  we demonstrate how our representation theory
relates to, and  illuminates,  results in the existing literature.  Section~\ref{Sec:Applications} is devoted to 
applications:
we exploit our natural dualities to establish a range of properties
of bilattices which are categorical in nature, for instance  
the determination of free objects and of unification~type.

We emphasise that our   approach 
differs in an
important  respect from that adopted by other authors. Bilattices have been very thoroughly  studied as algebraic structures 
 (see for example~\cite{RPhD} and the references therein).
 Central to the theory of distributive  bilattices, and more generally interlaced ones, is the theorem showing that such algebras 
can always be represented as products of pairs of lattices,  with the  structure 
determined from  the  factors 
(see \cite{P00} and \cite{BR11} for the bounded and unbounded cases, respectively, 
and the informative historical survey by Davey \cite{BD13} of the evolution of this oft-rediscovered result).
The product representation 
 is normally derived by performing quite 
extensive algebraic calculations.  
It is then used in a crucial way to obtain,  for those bilattice varieties which have bounded 
distributive lattice 
reducts,
 dual representations
which are based on Priestley duality \cite{MPSV,JR}.
Our starting point is different. For each class $\CA$  of algebras we study 
here and in \cite{CP2},   we first establish, by elementary arguments,  that  $\CA$ takes the form 
$\ISP(\M)$, where~$\M$ is  finite,  or, more rarely, 
$\ISP(\CM)$, where $\CM$ is a set of two finite algebras.  
(In \cite{CCP} we assume at the outset 
that $\CA$ is the quasivariety generated by  some finite algebra
in which we are interested.)
This gives us  direct access to the natural duality framework. 
From 
this perspective, the product representation is a consequence of the natural dual representation, and closely related to it. 
 For a  
reconciliation, in the distributive setting,
of  our  approach  and that of others and an 
 explanation of how these approaches differ, see 
Sections~\ref{sec:prodrep} and~\ref{Sec:Conclude}.

We may summarise as follows
what we achieve in this paper and 
in \cite{CCP,CP2}. 
For different varieties we call on different versions of the 
 theory of natural dualities.
Accordingly
our account  can, {\it inter alia},  
be read as a set of illustrated
tutorials on the natural duality methodology presented in a self-contained way.   
The examples we give will also be new 
to natural duality aficionados,   
but  for such readers we anticipate that the primary interest of our work 
will be  
its contribution 
to 
the understanding of the interrelationship between natural and Priestley-style dualities  
for finitely generated quasivarieties of distributive lattice-based algebras. 
For this we  exploit the piggybacking technique,  building on work initiated in our paper \cite{CPcop} and 
our constructions  elucidate precisely  
how product representations come about.
All our natural dual representations are new, as are our Priestley-style dual representations in 
the unbounded cases.  
Finally we draw attention to the
remarks with which we end the paper 
drawing  parallels  between the special role 
the knowledge order plays 
in our theory  and 
the role  this order  plays in
Belnap's semantics for a four-valued logic.

\section{Distributive pre-bilattices and bilattices} \label{sec:DBilat}

 We begin by giving  basic definitions and establishing the terminology we shall adopt henceforth.  
 We warn that  the definitions 
(bilattice, pre-bilattice, etc.)  are not used in a  consistent way in the literature, and that notation varies. 
Our choice of symbols   for lattice  operations  
enables us  
to  keep overt which operations relate to truth  and which to 
knowledge.  Alternative notation includes $\vee$ and $\wedge$ in place of $\lor_t$ and $\land_t$, and 
$\oplus$ and $\otimes$ in place of $\lor_k$ and~$\land_k$.  

We define first  the most general class of algebras we shall consider.
We shall say that an algebra 
$\A = (A; \lor_t,\land_t,\lor_k,\land_k)$ is an 
\defn{unbounded
 distributive pre-bilattice}  if each of the reducts 
$(A; \lor_t,\land_t)$ and $(A;\lor_k,\land_k)$ is a lattice 
and each of $\lor_t$, $\land_t$, $\lor_k$ and $\land_k$ 
distributes over each of the other three.  
The class of such algebras is a variety, which we denote by 
$\DPBU$.  Each of the varieties we consider in this paper and in \cite{CP2}
will be obtained  from $\DPBU$ by expanding the language 
by adding constants, or additional unary or binary operations.   

Given $\A \in \DPBU$, we let 
$\A_t = (A;\lor_t, \land_t)$ and refer to it as the \defn{truth lattice 
reduct} of $\A$ (or $t$-lattice for short);  likewise we have a \defn{knowledge lattice reduct} (or $k$-lattice)
$\A_k = (A;\lor_k,\land_k)$.

The following lemma is an elementary consequence of the definitions.
We record it here to emphasise that no structure 
beyond that of an unbounded distributive pre-bilattice is involved.

\begin{lem}  \label{lem:cong} 
 Let $\A= (A;\lor_t,\land_t, \lor_k,\land_k) \in \DPBU$.
 Then, for $a,b,c \in A$,
\begin{newlist}
\item[{\upshape (i)}] $ a \leq _k b \leq _k c $ implies
$ a \land_t c \leq_t  b \leq_ t a \lor_t c $;
\item[{\upshape (ii)}] $a \land_t b \leq_t a  \star_k b \leq_t a \lor_t b$,
where $\star_k$ denotes either $\land_k$ or  $\lor_k$.
\end{newlist}

Corresponding statements hold with $k$ and $t$ interchanged. 
\end{lem}

  As we have indicated in 
the introduction,
 we shall wish to prove, 
for each bilattice variety~$\CA$ we study, that 
$\CA$ is finitely generated  as a quasivariety. 
This amounts to showing that there exists a finite set $\CM$ of 
finite algebras in~$\CA$ such that, for each $\A\in \CA$ and $a \ne b $ in
$\A$, there is $\M \in \CM$ and a 
$\CA$-homomorphism $h \colon \A \to \M$ with $h(a) \ne h(b)$. 
($\CM$ will 
consist 
 of  a single subdirectly irreducible
algebra 
or at most two such algebras.)
 This separation property is linked to the  existence of particular quotients of the algebras in~$\CA$.
Accordingly we are led to investigate congruences.
We start with a known  result. 
 Our
 proof is direct and elementary:  it uses 
nothing more than  
the  distributivity properties of  the $t$- and $k$-lattice operations, together 
with
Lemma~\ref{lem:cong} and
basic facts about lattice congruences given, for example,
in \cite[Chapter~6]{ILO2}. 
(Customarily the lemma would be   obtained as a spin-off from the product representation theorem as this applies to distributive bilattices.)
 
\begin{prop} \label{lem:cong2} Let $\A = (A;\lor_t,\land_t, \lor_k, 
\land_k)$ be an unbounded  distributive pre-bilattice.
   Let $\theta\subseteq A^2$ be an equivalence relation.  Then the following statements are equivalent:
\begin{newlist}
\item[{\upshape (i)}] $\theta$ is a congruence of $\A_t =(A;\lor_t,\land_t)$;
\item[{\upshape (ii)}] $\theta$ is a congruence of $\A_k = (A;\lor_k,\land_k)$;
\item[{\upshape (iii)}]  $\theta$ is a congruence of 
 $\A$.
\end{newlist}
\end{prop} 
\begin{proof}  It will suffice, by symmetry, to 
prove (i) $\Rightarrow$ (ii).
 So assume that~(i) holds.  
Since 
$\theta$ is a congruence of 
$(A;\lor_t,\land_t)$,  
the $\theta$-equivalence classes are convex sublattices with respect to 
the $\leq_t$ order. We first observe that from Lemma~\ref{lem:cong}(i) each  equivalence class is convex 
with respect to the $\leq_k$ order,
and from 
Lemma~\ref{lem:cong}(ii) that each equivalence class is a sublattice of $(A;\lor_k,\land_k)$.   

Finally we need to establish the quadrilateral property:
  \[
   a\, \theta \, ( a\land_k b ) \Longleftrightarrow b \, \theta \,  (a \lor_k b).
\] 
For the forward direction  
observe that the distributive 
laws and  Lemma~\ref{lem:cong}(ii) (swapping $t$ and $k$) imply
\begin{multline*}
 a\land_t b=(a\lor_k b)\land_k(a\land_t b)=(a\land_t (a\land_k b))\lor_k (b\land_t (a\land_k b)) \\
=(a\land_k (a\land_t b))\lor_k (b\land_k (a\land_t b))
=(a\lor_k b )\land_k (a\land_t b).
\end{multline*} 
Combining this with $a\, \theta \, ( a\land_k b )$ and with the fact that $\theta$ is a congruence of $(A;\lor_t,\land_t)$, we have 
 $a\land_t b\, \theta \,  (a\lor_k b)\land_t a$. Replacing $\land_t$ by $\lor_t$
 in the previous argument, we obtain $a\lor_t b\, \theta \,  (a\lor_k b)\lor_t a$. 
This proves that
 \[
[a]_{\theta}\land_t [b]_{\theta}= [a]_{\theta}\land_t [a\lor_k b]_{\theta}\quad \mbox{and} \quad [a]_{\theta}\lor_t [b]_{\theta}= [a]_{\theta}\lor_t [a\lor_k b]_{\theta}.
\]
 Since 
 $(A;\land_t, \lor_t)/\theta$ is distributive, 
 $[b]_{\theta}=[a\lor_k b]_{\theta}$, that is, $b\,\theta\, a\lor_k b$.
\end{proof} 

The following consequences of Proposition~\ref{lem:cong2}
will be important  later. Take  an unbounded 
 distributive pre-bilattice $\A$ 
and a filter $F$ of $\A_t$.  Then   $F$ is a convex sublattice of $\A_k$. 
If a map $h\colon A  \to \{0,1\}$
acts as
 a lattice homomorphism from $\A_t$ into
the two-element lattice~$\two$, then $h$ is a lattice homomorphism from $\A_k$ into either~$\two$ or  its dual lattice $\two^{\partial}$. 
Hence  each prime filter for $\A_t$ is either a prime filter or a prime ideal for $\A_k$ and vice versa. These results were first proved in \cite[Lemma 1.11 and Theorem 1.12]{JR} and underpin the development of the duality theory
presented there.

We now wish to consider the situation in which a distributive pre-bilattice has universal bounds with respect to its $\leq_t$ and $\leq_k$ orders.
We  recall a classic result,  known as the $90^\circ$~Lemma. 
 The result 
has its origins in \cite{BK47} (see the comments in  \cite[Section 3]{JM}
and  
 also \cite[Theorem~3.1]{P00}).

 \begin{lem} \label{90deg}  Let $(L; \lor_t, \land_t, \lor_k,\land_k)$
be an unbounded  distributive pre-bilattice.  Assume 
  that 
$(L; \leq_k)$ 
has a bottom element, $0_k$,
and a top element, 
$1_k$.
\begin{newlist}  
\item[{\rm (i)}] 
For all $a,b \in L$, 
\begin{align*}
a \vee_k b &= ((a \wedge_t b)\wedge_t 0_k 
) \vee_t ((a \vee_t b)\wedge_t 1_k 
),\\
a \wedge _k b &= ((a \wedge_t b)\wedge_t 
1_k 
) \vee_t ((a  \vee_t b)\wedge_t 0_k 
).
\end{align*}
\item[{\rm (ii)}] For all $a \in L$,
\[
0_k \land_t  1_k 
 \leq_t a \leq_t 0_k \lor_t 1_k, 
\]
so that $(L,\leq_t)$ also has universal bounds, and in the lattice 
$(L;\lor_t,\land_t)$, the elements 
$0_k$ and $1_k$ 
form a complemented pair. 
  \end{newlist}
\end{lem} 
The import of Lemma~\ref{90deg}(i)
is  
that $\lor_k$ and $\land_k$ are term-definable from 
$\lor_t$ and $\land_t$ and the universal bounds of the $k$-lattice; henceforth when these
 bounds are included in the type we  shall exclude $\lor_k$ and $\land_k$ from it.
When we refer to an algebra $\A = (A;\lor_t,\land_t,\lor_k,\land_k)$
as 
being an unbounded distributive pre-bilattice we do not exclude 
the possibility that one, and hence both, of $\A_k$ and~$\A_t$ has universal bounds; we are simply saying that bounds are 
not included in the algebraic language. 
We 
say an algebra $(A; \lor_t,\land_t,
0_t,1_t,0_k,1_k)$
is a 
\defn{distributive pre-bilattice}
if   $0_t$, $1_t$, 
$0_k$ and~$1_k$
 are nullary operations, and the algebra 
$(A; \lor_t,\land_t,\lor_k,\land_k)$ belong to $\DPBU$, where 
$\lor_k$ and  $\land_k$ are defined from $\lor_t$, $\land_t$,~$0_k$ and~$1_k$ as in Lemma~\ref{90deg}(i), 
and $0_t$, $1_t$ and 
$0_k$, $1_k$
act as  $0$, $1$  in the lattices
$\A_t$ and~$\A_k$, respectively.  

We now  add 
a negation operation.    If 
${\A =
(A;\lor_t,\land_t, \lor_k, \land_k)}$  belongs to $\DPBU$ and 
carries an involutory unary operation $\neg$ which 
is interpreted as a dual  endomorphism 
of $(A; \lor_t,\land_t)$ and an endomorphism of 
$(A; \lor_k,\land_k)$, then we  
 call 
$(A; \lor_t,\land_t,\lor_k,\land_k,\neg)$ an 
\defn{unbounded distributive bilattice}.   
Similarly, 
an algebra 
$(A;\lor_t,\land_t, 
\neg, 0_t,1_t,0_k,1_k)$ is a \defn{distributive bilattice} if the negation-free 
reduct is a distributive pre-bilattice, 
 and $\neg$ is an involutory  dual
endomorphism of the bounded $t$-lattice reduct and   
endomorphism 
of the  bounded $k$-lattice reduct.  These conditions include 
the requirements that~$\neg$ interchanges $0_t$ and~$1_t$ and fixes $0_k$ and $1_k$.

For ease of reference we present a list of the varieties we consider 
in this paper,  in the order in which we shall study them. 
\begin{longnewlist}
\item[$\DB$:]  {\bf distributive bilattices}, for which  we include in the type 

$\lor_t$, $\land_t$, $\neg$, $0_t$, $1_t$, $0_k$, $1_k$;
\item[$\DBU$:]  {\bf  unbounded distributive bilattices}, having as 
basic operations

 $\lor_t$, $\land_t$, $\lor_k$, $\land_k$, $\neg$; 
\item[$\DPB$:] {\bf distributive pre-bilattices}, 
having as basic operations 

$\lor_t$, $\land_t$, $0_t$, $1_t$, $0_k$, $1_k$;
\item[$\DPBU$:] {\bf unbounded distributive pre-bilattices},
having as basic operations
$\lor_t, \land_t$, $\lor_k$, $\land_k$.
\end{longnewlist}

We shall denote by $\CCD$ the variety of distributive lattices 
in which universal bounds are included in the type, and by $\DU$
the variety of unbounded distributive lattices.  
For any $\A \in \DB$ or $\DPB$,  
its bounded truth lattice $\A_t = (A; \lor_t,\land_t, 0_t,1_t)$ 
is a 
 $\CCD$-reduct of $\A$.
 Likewise the truth lattice $\A_t = (A;\lor_t,\land_t)$ 
provides a reduct in
$\DU$
for any $\A \in \DBU$ or $\DPBU$.
 We remark also 
that each member of $\DB$
has a  reduct in the variety~$\cat{DM}$ of De Morgan algebras, 
and that each algebra  in $\DBU$ has a reduct in the variety of 
De Morgan lattices;  in each case the reduct is obtained by suppressing the knowledge operations. 
This remark explains the preferential treatment  
we always give to truth over knowledge when forming reducts.

Throughout 
we shall  
when required treat  any variety 
as a category, by taking as morphisms  all 
homomorphisms. 
 Given a variety~$\CA$ whose algebras have 
reducts (or more generally term-reducts)
in $\CCD$ 
 obtained by deleting certain operations, 
we shall make use of  the  associated forgetful
functor from $\CA$ into $\CCD$, 
defined to act as the identity map on morphisms.   
(We shall later refer to $\CA$ as being \defn{$\CCD$-based}.)
Specifically  we define a forgetful functor $\U \colon \DB \to \CCD$,
for which  
 $\U(\A)=\A_t$ for any $\A \in \CCD$.
We also have a functor, again  denoted~$\U$ and defined in the same way,
from $\DPB$ to $\CCD$. 
 Likewise there is 
a functor
$\Uu$ from  $\DBU$
or from $\DPBU$ into $\DU$ 
 which sends an algebra to its truth lattice.

\begin{figure}  [ht]
\includegraphics[scale=1, trim=26  670  26 50]{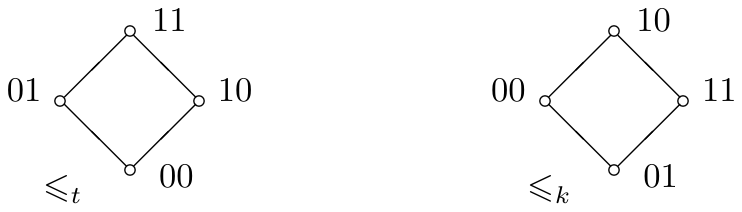}
\caption{The $t$- and $k$-lattice reducts of $\four$ and $\fourU$} \label{strings}
\end{figure}

We now  recall the best-known 
(pre-)bilattice of all, 
 that known as 
$\mathcal{FOUR}$.
We 
consider the set $V=\{0,1\}^2$ and,
to simplify later notation, 
shall  denote its
elements by binary strings.  
We define  lattice orders $\leq_t$ and $\leq_k$ on $V$  as shown in Figure~\ref{strings};  
 we draw lattices  in the 
manner traditional in lattice theory.  
(In the literature of  bilattices,  the four-element 
pre-bilattice  is customarily depicted via an amalgam  of the lattice diagrams in Figure~\ref{strings}, with the two orders indicated vertically (for knowledge)
and horizontally (for truth); 
virtually every paper  on bilattices contains this figure and we do not reproduce it here.)

We may add 
 truth constants 
$0_t = 00$ 
and  $1_t = 11$ and knowledge constants $0_k = 01$
and $1_k = 10$ to $\mathcal{FOUR}$ to 
obtain a member of $\DPB$. 
The structure 
$\mathcal{FOUR}$  also supports a negation $\neg$ which switches $11$ and $00$ 
and fixes $01$ and $10$.  
The four-element distributive bilattice and its unbounded counterpart play a distinguished role in what follows.  Accordingly we 
 define  
\begin{align*} 
\four & = (\{ 00,11,01,10\}; \lor_t, \land_t, \neg, 0_t,1_t,0_k,1_k) \text{ and} \\
\fourU  &= (\{00,11,01,10\}; \lor_t, \land_t,\lor_k,\land_k, \neg).
\end{align*} 
These belong, respectively, to $\DB$ and to $\DBU$.

There are two non-isomorphic  two-element 
 distributive pre-bilattices without bounds.
 One, denoted $\twoU^{+}$,  has underlying set 
$\{0,1\}$, and  the $t$-lattice structure and the $k$-lattice 
structure both coincide with that of the two-element lattice 
$\two= (\{ 0,1\}; \lor, \land)$ in which $0 < 1$.  
The  other, denoted  $\twoU^{-}$,  has $\two$ as its $t$-lattice reduct and the order dual
$\two^\partial$ as its $k$-lattice reduct.   
If we include bounds, we must have 
$0_t = 0_k = 0$ and 
$1_t = 1_k=1$ 
if $\leq_k$ and $\leq_t$ coincide 
 and 
$0_t = 1_k = 0$ and 
$1_t = 0_k=1$ if $\leq_k$ coincides with $\geq_t$.

In neither the bounded nor the unbounded case do we have a two-element
algebra which supports an involutory negation which 
preserves $\land_k$ and $\lor_k$ and interchanges  
$\lor_t$ and $\land_t$.
Hence   neither $\DBU$ nor 
$\DB$  
contains a two-element algebra.  Similarly, if 
either variety contained a three-element algebra, having  universe $\{ 0,a,1\}$, 
with $0 <_t a <_t 1$, then~$\leq_k$ would have to coincide
with either  $\leq_t$ or $\geq_t$.   The only 
 involutory dual endomorphism of the 
$t$-reduct of the chain
swaps $0$ and $1$ and fixes $a$, and this map is not order-preserving with respect to~$\leq_k$. 
 We conclude that, whether or not bounds are included in the type, there is no non-trivial distributive bilattice of cardinality less than
four.  
  Hence, the $90^\circ$
Lemma implies that $\four$ and
$\fourU$ are the only four-element
 algebras in $\DB$ and $\DBU$, respectively.

As noted above, 
to
derive
a natural duality 
for any  one 
of the varieties  in which we are interested, we need to
express the variety $\CA$ in question
as a finitely generated quasivariety.  Specifically,
 we need to find  a finite set 
$\CM$  of finite algebras  such that
 $\CA = \ISP(\CM)$.  We shall prove  in subsequent sections,
with the aid of  Proposition~\ref{lem:cong2},
that 
\begin{alignat*}{2}
 \DB  &= \ISP(\four), \qquad  \qquad & \DPB &= \ISP(\two^+,\two^-),\\ 
\DBU &= \ISP(\fourU),   & \DPBU & = \ISP(\twoU^+,\twoU^-).
\end{alignat*}
Corresponding results hold for
the varieties we consider in \cite{CP2}.  
Such
results  are  central to our enterprise. 
All are elementary in that the proofs use  
a minimum of
bilattice theory  and   
none of the algebraic structure theorems  for bilattices  is needed.  
(There is  a close connection between our assertions above and 
the identification of the subdirectly irreducible algebras
 in the varieties concerned. 
 The latter 
has traditionally been handled 
by first
proving a product representation theorem.  We reiterate that
we prove our claims directly, by elementary~means.)

\section{The natural duality framework}\label{piggyonesorted}
As indicated in Section~\ref{sec:intro}, we shall introduce
natural
 duality machinery in the form that is simplest to apply to  each of the 
varieties we consider. 

We first  
consider $\CA = \ISP(\M)$, where $\M$ is a finite algebra with a lattice reduct.
 We shall aim to define an \defn{alter ego} $\MT$ for $\M$ which will serve to generate
a category $\CX$ dually equivalent to~$\CA$.
The alter ego  will be  
a discretely topologised structure $\MT$ on the same universe $M$ as~$\M$ 
and will be 
equipped with a set $R$ of relations 
which are \defn{algebraic} in the sense that each member  
of $R$ is a 
subalgebra of  some finite power $\M^n$ of~$\M$. 
 (Later we shall need also to allow for nullary operations,
but relations
suffice in the simplest cases we consider.)
We define~$\CX$ to be the topological quasivariety
$\IScP (\MT)$ generated by~$\MT$, that is, the class of isomorphic copies of closed substructures
of non-empty powers of $\MT$; the empty structure is 
included.
 The structure of the alter 
ego is lifted pointwise in the obvious way.  
We denote the lifting of $r \in R$ to a member $\X $ of $ \CX$ by $r^{\X}$.
We then have well-defined
contravariant  
 functors $\D \colon \CA \to \CX$ and $\E \colon 
\CX \to \CA$ defined as follows:
\begin{alignat*}{3}
&\text{on objects:} & \hspace*{2.5cm}  &\D \colon  \A \mapsto  \CA(\A,\M), 
  \hspace*{2.5cm}  \phantom{\text{on objects:}}&&\\
&\text{on morphisms:}  & &  \D \colon  x \mapsto - \circ x,&& \\
\shortintertext{where  $\CA(\A,\M)$  is seen as
 a closed
 substructure of $\MT^{\A}$,  and}
&\text{on objects:} & & \E  \colon  \X \mapsto  \CX(\X,\MT),
\phantom{\text{on objects:}}&&\\
&\text{on morphisms:}  & &\E  \colon  \phi \mapsto - \circ \phi,
\phantom{\text{on morphisms:}}&&
\end{alignat*}
 where 
 $\CX(\X,\MT)$ is  seen 
as a
subalgebra of $\M^{\X}$.

 Given  $\A \in \CA$, 
we shall refer to $\D(\A)$ as the \defn{natural dual} of $\A$.   
 We have, for each~$\A \in \CA$, a  natural evaluation map 
$e_{\A}\colon \A \to \ED(\A)$, given by 
$e_{\A}(a)(x) = x(a)$ for $a \in A$ and $x \in \D(\A)$, and likewise 
there exists  an   
 evaluation map $\varepsilon_{\X }\colon \X \to \DE(\X)$ for 
$\X \in \CX$.  
We say that 
\defn{$\MT$ yields a duality on $\CA$} if  
$e_{\A}$ is an isomorphism for each $\A \in \CA$,  and that
\defn{$\MT$ yields a full duality on~$\CA$} if  in addition 
$\varepsilon_{\X}$ is an isomorphism for each $\X \in \CX$.  
Formally, if we have a full duality then $\D$ and $\E$ set up a dual equivalence between $\CA$ and $\CX$  with the unit and co-unit of the adjunction given by the evaluation maps.  
All the dualities we shall present in this paper are full and, moreover, in each case
we are able to give a precise description of 
the dual category~$\CX$.  
Better still, the dualities have the property that 
they are strong dualities. 
For the definition of a strong duality  and a full discussion of this notion
 we refer the reader
to \cite[Section~3.2]{CD98}.  Strongness
implies that 
$\D$ takes injections to surjections
and surjections to embeddings, 
facts which we shall exploit in 
Section~\ref{Sec:Applications}.

Before proceeding  we indicate, 
for the benefit of readers not 
conversant with natural duality theory, how 
 Priestley duality  
fits into this framework.
We have 
\allowdisplaybreaks
\begin{alignat*}{2}
\CA &= \CCD, \qquad
 && \text{the class of distributive lattices with $0,1$}, \\
\M &= \two,  \quad  &&\text{the two-element chain in $\CCD$}; \\
\CX &= \CP,  \quad &&\text{the category of Priestley spaces}, \\
\MT &= \twoT, \quad && \text{the discretely topologised two-element chain}; \\
R &= \{ \leq\},  \quad  && \text{where $\leq
$ is the subalgebra $\{ (0,0), (0,1), (1,1)\}$ of
$\two^2$}.    
\end{alignat*}
This duality is strong   \cite[Theorem 4.3.2]{CD98}.   We 
later exploit it as a tool when dealing with bilattices  having reducts in~$\CCD$ and
 it is 
convenient henceforth to 
denote the hom-functors
$\D$ and $\E$ setting it up by~$\fnt{H}$ and $\fnt{K}$.
When expedient, 
we
view~$\KH(\Lalg)$ as the family of clopen up-sets of~$\Lalg$, for
$\Lalg \in \CCD$.

In accordance with our black-box philosophy  we shall present 
without further preamble
the first 
of the duality theorems we shall use.  
 It addresses both  the issue of the existence 
of an alter ego  yielding  
a duality  and that of finding one which is 
conveniently simple. Theorem~\ref{genpigoneM} comes from 
specialising \cite[Theorem~7.2.1]{CD98}  and  the fullness assertion from 
\cite[Theorem~7.1.2]{CD98}. 

 We  deal  with a  
quasivariety of algebras $\CA$ generated by an algebra~$\M$ with a 
reduct in~$\CCD$ and 
denote by $\U$ the associated  forgetful
 functor from~$\CA$ into~$\CCD$. 
For $\w_1,\w_2 \in \Omega = \CCD(\U (\A),\two)$,
we let 
  $R_{\omega_1,\omega_2}$ be  the collection  of
maximal $\CA$-subalgebras
 of sublattices of the form
\[(\omega_1,\omega_2 )^{-1}(\leq) =
\{\,  (a,b) \in \M^2 \mid \omega_1(a) \leq \omega_2 (b)\,\}. 
\]

\begin{thm} \label{genpigoneM} {\rm(Piggyback Duality Theorem  for $\CCD$-based algebras, single generator  case)}
Let 
$\CA = \ISP (\M) $,
where 
$\M $ is a finite algebra with a 
reduct in $\CCD$,  
and  
 ${\Omega=\CCD(\U (\A),\two)}$.
 Let $\MT = (M  ; R, \Tp)$
 be the topological relational structure
on  the underlying set~$M$ of  $\M$ in which 
$\Tp $ is the discrete topology and 
 $R$ 
 is  the union  of  the sets   $R_{\omega_1,\omega_2}$ as 
 $\omega_1,\omega_2$ run over  $ \Omega$.
Then $\MT$ yields a natural duality on $\CA$.

Moreover,  if 
$\M $ is subdirectly irreducible,  has no proper subalgebras
and no endomorphisms other than the identity,  
then $\MT$ as defined above determines a strong  duality.  
So
the functors $\D= \CA(-, \M)$ and $\E = \CX(-, \MT)$  set up  
a dual equivalence between $\CA= \ISP(\M)$ and $\CX=\IScP(\MT)$. 
\end{thm} 

We now turn to the study of algebras 
 which have reducts in $\DU$ rather than in~$\CCD$.  We 
consider a class $\CA$ of algebras for which we have a 
forgetful functor 
$\Uu $ from $\CA$ into~$\DU$.
The natural duality for~$\DU$ will 
take the place of Priestley duality for 
$\CCD$.
 This duality  is less well known to those who are not 
specialists in duality theory, 
but it is equally simple.   
We have $\DU = \ISP(\twoU)$,  where  $\twoU= (\{ 0,1\}; \land,\lor)$.
The alter ego is $\twiddle 2_{01} =(\{0,1\}; 0, 1, \leq, \Tp)$,
where $0$ and  $1$ are treated as nullary operations. 
It yields a strong duality between~$\DU$ and 
the category~$\CP_{01} = \IScP(\twiddle 2_{01} )$ of 
doubly-pointed Priestley spaces (bounded Priestley spaces in 
the terminology of~\cite[Theorem~4.3.2]{CD98}, 
where validation of the strong duality can also be found).  
 The duality is set up by  well-defined hom-functors 
$\Hu = \DU(-,\twoU) $ and $\Ku = \CP_{01} (-, \twiddle 2_{01})$.  
A member $\Lalg$ of $\DU$ is isomorphic  to  $\Ku\Hu(\Lalg)$ and  
may  be identified with the lattice of proper non-empty clopen 
up-sets of the doubly-pointed Priestley space $\Hu (\Lalg)$.

 Most previous  applications of the piggybacking theory have been  made over~$\CCD$ (see \cite[Section~7.2]{CD98}), or over the variety of unital semilattices. 
But one can 
 equally well 
piggyback 
over~$\DU$; see \cite[Theorem~2.5]{DP87} and 
\cite[Section~3.3 
and 
 Subsection~4.3.1]{CD98}.   
 (In \cite{CP2} we extend the scope further: we handle bilattices with conflation by piggybacking over $\DB$ and $\DBU$.)

\begin{thm} \label{genpigoneMu} {\rm(Piggyback Duality Theorem  for 
$\DU$-based algebras, single generator case)} 
 Suppose that 
$\CA = \ISP (\M) $, where 
$\M $ is a finite algebra with a  
reduct in
$\DU$ but no 
reduct  in~$\CCD$.  
Let  
$\Omega=\DU(\Uu(\M),\twoU)$ 
and 
$\MT = (M  ; R, \Tp)$
 be the topological relational structure
on  the underlying set~$M$ of  $\M$ in 
which~$\Tp $ is the discrete topology and 
$R$ contains  the  
relations of 
 the following types:
\begin{newlist}
\item[{\rm (a)}]  the members of  the sets   $R_{\w_1,\w_2}$, as 
 $\w_1,\w_2$ run over  $ \Omega$, where 
  $R_{\w_1,\w_2}$ is the  set 
 of 
maximal $\CA$-subalgebras
 of sublattices of the form
\[(\w_1,\w_2 )^{-1}(\leq) =
\{\,  (a,b) \in \M^2 \mid \w_1(a) \leq \w_2 (b)\,\}; 
\]
\item[{\rm (b)}] the members of the sets $ R^i_{\w}$, as
 $\w$ runs over  $ \Omega$ and $i\in\{0,1\}$, where  $R^i_{\w}$ is the set 
 of maximal $\CA$-subalgebras
 of sublattices of the form
\[
\w^{-1}(i) =
\{\,  a\in \M \mid \w(a)=i\,\}.
\]
\end{newlist}
Then $\MT$ yields a natural duality on $\CA$.

Assume moreover  that $\M$ is subdirectly irreducible,
that $\M$ has no
non-constant 
endomorphisms 
other than the identity on $\M$ and that the only 
proper subalgebras of~$\M$ are one-element subalgebras.  
Then 
the duality above  can be 
upgraded to a strong, and hence full,  duality  by  including in the alter ego~$\MT$ all one-element subalgebras
of~$\M$,  regarded  as nullary  operations. 
If $\CX=\IScP(\MT)$, where $\MT$ is upgraded as 
indicated, 
then the functors $\unbounded{\D}= \CA(-, \M)$ and~$\unbounded{\E}= \CX(-, \MT)$ yield a dual equivalence between $\CA$ and $\CX$.  
 \end{thm} 

\begin{proof}    
Our claims regarding the duality follow from 
\cite[Section~2]{DP87}.   For a discussion of the role played by 
the nullary operations in yielding a strong
duality, we refer the reader
 to \cite[Section~3.3]{CD98}, noting that 
our assumptions on $\M$ ensure that any 
non-extendable partial endomorphisms would have to  
have one-element domains. 
Hence
it suffices to include these one-element subalgebras 
as nullary operations in order to obtain a strong duality. 
\end{proof}

We conclude this section with 
remarks on the special role of
piggyback dualities.  For  
quasivarieties to which either 
Theorem~\ref{genpigoneM} or Theorem~\ref{genpigoneMu} 
applies, we could have taken a different approach, based on the NU Strong Duality Theorem \cite[Theorems~2.3.4 and~3.3.8]{CD98},
as it applies to a quasivariety $\CA= \ISP(\M)$, where 
$\M$ is a finite algebra with a lattice reduct.
This way, the set of piggybacking subalgebras would have been replaced 
by the set of all subalgebras of $\M^2$.  But this has two disadvantages, 
one well known, the other revealed by our work in \cite[Section~2]{CPcop}.      
Firstly, the set of all subalgebras of $\M^2$ may be unwieldy, even when $\M$ 
is small.  In part to address this, a theory of entailment has been devised, 
which allows superfluous relations to be discarded from a duality; 
see \cite[Section~2.4]{CD98}.  The piggybacking method, by contrast, provides 
alter egos  which are  much closer to being optimal.  Secondly, as we 
reveal 
in Section~\ref{Sec:DisPiggyDual}, the piggyback relations play a special role
in translating natural dualities to ones based on the Priestley dual spaces of 
the algebras in $\U(\CA)$ or $\Uu(\CA)$, as appropriate.  We shall also 
see that, even when certain piggyback relations can be discarded from an alter ego without destroying the duality, these relations do  make a contribution  in the translation   process.

\section{A natural duality for 
  distributive bilattices} \label{sec:DB} 
In this section we set up  a duality for  
 the variety $\DB$ and reveal the special role played on the dual side by the knowledge order.

\begin{prop}  \label{sep-prop-bdd}
$\DB = \ISP(\four)$.
\end{prop} 

\begin{proof} 
Let $\A \in \DB$.  Let $a \ne b $ in $\A$ and choose $x \in \CCD
(\A_t,\two)$ such that ${x(a) \ne x(b)}$.
  Define an equivalence 
relation~$\theta$ on $\A$ by $p \, \theta \, q$ if and only if 
${x(p) = x(q)}$ and $x(\neg p ) = x(\neg q)$.  Clearly~$\theta$ is a congruence of
$\A_t$.  By Proposition~\ref{lem:cong2} it is also a congruence of 
$\A_k$, and by its definition it preserves~$\neg$.  In addition,
$A/\theta$ is a non-trivial algebra (since $x(a) \ne x(b)$) and
of cardinality at most four.  Since the only such algebra
in $\DB$, up to isomorphism, is $\four$, the image of the associated  
$\DB$-homomorphism $h \colon \A \to \A/\theta  $ is (isomorphic to)~$\four$, and separates~$a$ and~$b$.
\end{proof}

 It is instructive also to 
present 
$h
\colon \A \to \four$, as above, 
more directly.  
 We take
\[
h(c) =\begin{cases}
x(c) (1-x(\neg c)) &\text{if }  x(0_k) = 0, \\
          (1-x(\neg c))x(c)&\text{if } x(0_k) = 1,
           \end{cases}
\]
for all $c$; 
here we are viewing the image $h(c)$ as a binary string. 
In the case that $x(0_k) = 0$,
 observe that $h(0_k) = 01 =0_k^{\four}$ (note that $\neg 0_k = 0_k)$).  Since 
$x(0_k)\wedge x(1_k)=x(0_k \land _t 1_k) = x(0_t) =0$ and
 $x(0_k)\vee x(1_k)=x(0_k \lor _t 1_k) = x(1_t) =1$,  
we have   
$x(1_k) =x(\neg 1_k) = 1$ and $h_0(1_k)=10=1_k^{\four}$. 
It is routine to check 
that $h$ preserves 
$\lor_t$, $\land_t$ and~$\neg$. 
Hence~$h$  is a $\DB$-morphism and, by construction, $h(a) \ne h(b)$.
The argument for the case that $x(0_k) = 1$ is similar.

In the following result  we 
make use of the $\CCD$-morphisms from
the $t$-lattice reduct  of~$\four$ into $\two$.
These are the maps 
 $\alpha$ and $\beta$ 
 given respectively by 
$\alpha^{-1}(1) = \{ 10, 11\}$  and ${\beta^{-1}(1) = \{ 01, 11\}}$. 
Observe that $\alpha$ and $\beta$ correspond to the maps 
that assign  
to a binary string its first and second elements, respectively.

\begin{thm} \label{DBnatdual}   {\rm  (Natural duality for 
distributive bilattices)} 
There is a dual equivalence between the category $\DB$ and the category 
$\CP$ of Priestley spaces set up by hom-functors.  Specifically,~let 
\[
\four=  \bigl( \{ 00, 11,01,  10\}; \lor_t, \land_t,  
\neg, 0_t, 1_t, 0_k, 1_k\bigr)
\]
 be the four-element bilattice in the variety 
$\DB$ of distributive bilattices
and let its alter ego be
\[
\fourDBT = \bigl(  \{ 00, 11, 01,10\};  \leq_k, \Tp\bigr).
\]
Then  
\[
\DB = \ISP(\four) \quad \text{and} \quad \CP = 
\IScP (\fourDBT)
\]
 and  the  hom-functors  
$\D = \DB(-, \fourDB)$ and 
$\E = \CP(-, \fourDBT)$ set up a dual equivalence 
between $\DB$ and $\CP$. Moreover, this duality is strong.
\end{thm}

\begin{proof}  The  proof  involves three 
steps.  

\noindent{\bf Step 1:  setting up the piggyback duality.}\newline
We must 
identify the subalgebras of 
 $\four^2$ involved in the piggyback duality supplied by
Theorem~\ref{genpigoneM} when  $\CA = \DB$ and $\M = \four$.  
Define $\alpha$ and $\beta$ 
as above. 
 
We claim that 
the knowledge order 
 $\leq_k$ is the unique maximal $\DB$-subalge\-bra
of $(\alpha,\alpha)^{-1}(\leq)$. 
We first observe that 
 it is immediate from order properties of lattices
that $\leq_k$ is a sublattice for the 
$k$-lattice structure.  
It also contains
the elements 
$01 \, 01$ and $10\, 10$. 
 By the $90^\circ$ Lemma (with~$k$ and $t$ switched), $\leq_k$ is also closed under 
$\land_t$ and $\lor_t$ (or this can be easily checked directly).
Since~$\neg$ preserves $\leq_k$, we conclude that $\leq_k$ is  a subalgebra of $\four^2$.

Now note  that, for $a = a_1a_2$ and $b = b_1b_2$ binary strings in $\four$, we have
$\alpha(a) \leq \alpha(b)$ if and only if $a_1\leq b_1$ and that 
$\alpha (\neg a) \leq \alpha(\neg b)$ if and only if $1-a_2 \leq 1-b_2$
that is, if and only if $b_2 \leq a_2$.  
It follows that if $(a,b)$ belongs to 
a $\DB$-subalgebra of $(\alpha,\alpha)^{-1}(\leq)$ then
$(a,b)$ belongs to 
the relation $\leq_k$. 
 Since we have already proved that $\leq_k$ is a $\DB$-subalgebra of
 $(\alpha,\alpha)^{-1}(\leq)$ we deduce that $\leq_k$ is the unique maximal  subalgebra contained in this sublattice.  Likewise,
the unique maximal $\DB$-subalgebra of $(\beta,\beta)^{-1}(\leq)$ 
is $\geq_k$. 

We claim that no subalgebra of $\four^2$ is  contained in
$(\alpha,\beta)^{-1}(\leq)$. 
 To see this 
we observe that  $\alpha(0_k) = \alpha(10)= 1 \nleqslant
0 = \beta(10) = \beta(0_k)$.  Likewise, consideration of $1_k$
shows that there is no $\DB$-subalgebra contained in 
$(\beta,\alpha)^{-1}(\leq)$.

Following the Piggyback Duality Theorem slavishly, we should include
both 
$\leq_k$ and~$\geq_k$ in our alter ego.  But 
it is never necessary to include  a binary relation and also its converse in an alter ego, so  $\leq_k$ suffices.  

\noindent{\bf Step 2: describing the dual category.}\newline
To prove that $\IScP(\fourDBT)$ is the category of Priestley spaces 
it suffices to note  that $\twoT \in \ope{I}\Su_c (\fourDBT)$ and that 
$\fourDBT \in \ope{IP}(\twoT)$.
It follows that
$\IScP (\twoT) \subseteq \IScP(\fourDBT)$ and 
$\IScP(\fourDBT) \subseteq \IScP(\twoT)$.    

\noindent{\bf Step 3: confirming the duality is strong.}\newline 
We verify  that 
the sufficient conditions given in Theorem~\ref{genpigoneM} for the duality to be 
strong
are satisfied by $\M = \four$. 
We proved in Section~\ref{sec:DB} that there is no non-trivial algebra in 
$\DB$ of cardinality less than four.  Hence $\four$ has no non-trivial quotients and no 
proper subalgebras.  This  implies, too, that
$\four$ is subdirectly irreducible.
Since every element of $\four$ is the interpretation of a nullary operation, 
the only endomorphism of~$\four$ is the identity.
\end{proof}

 We might wonder whether there are alternative choices for the structure of the alter ego $\fourDBT$ of $\four$.  We now demonstrate  that, 
within the realm of binary algebraic relations at least,  there is 
no alternative:
it is inevitable that the alter ego contains the relation~$\leq_k$  (or its converse).

\begin{prop} \label{DBsub}
The subalgebras of $\four^2$ 
are
$ \four^2$,  $ \Delta_{\four^2}$, $\leq_k $ and  $\geq_k$.
Here $\Delta_{\four^2} $ denotes the  diagonal subalgebra $\{ \, (a,a) \mid a \in \four\,\}$. 
\end{prop}

\begin{proof} 
We merely outline the proof, which is  routine, but  tedious. 
Assume we have a  proper subalgebra $r$ of $\four^2$, necessarily containing $\Delta_{\four}$ (since all the elements of $\four$ are constants in the language of $\DB$) and assume that $r 
$ is not~$\leq_k$.  We must then check that~$r$ has to be
$\geq_k$.  
The proof relies on  
 two facts: (i)  
an element  belongs to
$r$ if and only if its negation does  and (ii) 
 if
$a = b \star c$, where $\star \in  \{ \lor_t,\land_t,\lor_k,\land_k\}$ and $c \in  r$, then 
$a \notin  r$ implies $b \notin r$.
\end{proof}

  The proposition  allows us, if we prefer, to arrive 
at Theorem~\ref{DBnatdual}  without recourse to the piggyback method.  
As  noted at the end of Section~\ref{piggyonesorted},
 it is  possible 
to obtain a duality  for  a finitely generated lattice-based quasivariety 
$\CA=\ISP(\M)$ by including in the alter ego  all subalgebras of~$\M^2$.
 Applying this to 
$\DB=\ISP(\four)$,
we obtain a duality
 by equipping the alter ego with the four 
relations listed in Proposition~\ref{DBsub}.  The subalgebras
$\four^2$ and $\Delta_{\four^2}$  
qualify as `trivial relations'  
and  can be discarded   
and we need only one of 
$\leq_k$ and $\geq_k$; see \cite[Subsection~2.4.3]{CD98}.  
Therefore the piggyback duality we presented earlier is essentially the only natural duality based on binary algebraic relations.  (To have included
relations of higher arity instead would have been possible, but
would have produced a duality which is essentially the same, 
but artificially complicated.)
We remark 
that the situation for $\DB$ is atypical, thanks to the very rich algebraic structure of $\four$.

\section{A natural duality for unbounded distributive bilattices}
\label{sec:DBU} 
  We now focus 
 on the variety $\DBU$, to which we shall apply   
Theorem~\ref{genpigoneMu}.
 We first need to 
represent $\DBU$ as a finitely generated quasivariety.

\begin{prop} \label{sep-prop-unbdd}
$\DBU = \ISP(\fourU)$.
\end{prop}
\begin{proof}  We take $\A \in \DBU$ and $a \ne b $ in $\A$ and use 
the Prime Ideal Theorem for unbounded distributive lattices to
find $x \in \DU(\A_t,\twoU)$ with    
$x(a) \ne x(b)$.
We may then argue exactly as we did in the proof of Proposition~\ref{sep-prop-bdd}, but now using the fact that $\fourU$ is, up to isomorphism, 
the only non-trivial algebra in $\DBU$ of cardinality at most four.
\end{proof}

We are                                                                                                                                                                                                 
ready to embark on  setting  up a piggyback
duality for $\DBU$.  
We find the piggybacking relations 
by  drawing on 
the description of  $\Su (\four^2)$ given in Proposition~\ref{DBsub} 
to describe  
$\Su(\fourU^2)$.  As a byproduct, we shall see that
among dualities 
whose 
alter egos contain relations which are at most binary, 
the knowledge order plays a distinguished role, just as it does 
in the duality for~$\DB$.  

Below, 
to simplify the notation,  
 the elements of~$\four^2$
are  
 written as 
 pairs of 
binary strings. 
For example, 
 $01\, 11$ is our shorthand for 
$(01,11)$.

\begin{prop}  \label{fourUsubs}
The subalgebras of $\fourU^2$ 
are of two types:  
\begin{newlist}
\item[{\rm (a)}]
the subalgebras of $\four^2$, as identified in 
Proposition~{\upshape\ref{DBsub}};
\item[{\rm (b)}] decomposable subalgebras, in which each factor 
 is  
$\{01\}$, $\{10\}$ or $\fourU$.
\end{newlist}
\end{prop}
 
\begin{proof}  
The    subalgebras  of $\fourU$ are  $\{ 01\}$,   $\{10\}$
and $\fourU$.  Any  indecomposable  subalgebra of 
$\fourU^2$ must then be such that  
the projection maps onto each coordinate have image $\fourU$.
We claim that any 
 indecomposable $\DBU$-subalgebra~$r$ of  $\fourU^2$ 
is a $\DB$-subalgebra of $\four^2$.
Suppose that $r\ne \Delta_{\fourU^2}$, the diagonal subalgebra of  
$\fourU^2$,  
and~$r$ is indecomposable.
Then~$r$ 
would contain elements
$a\, 01$ and $a'\, 10$
 for some $a, a'\in\fourU$.
 If $a=01$ and $a'=10$. 
Then  
$11\, 11$  and $00\, 00$ are in $r$  
and 
hence~$r$ is a subalgebra of $\four^2$. 
If  $a\ne 01$, then also 
$(a \land_k \neg a) \, 01 \in r$.  
Any
of the possibilities  $a = 00, 11,01$ implies  that
$10\, 01 \in r$.
  Therefore we must have 
$10\, 01 \in r$
 and likewise
$01\, 10 \in r$.    
Then, considering $\lor_t$ and $\land_t$, we get that
$11\, 11$ and $00\, 00$ are in $r$.    
But this implies 
$01\, 01\in r$,
by considering~$\land_k$.  Similarly 
$10\, 10 \in r$.  
The case $a' \ne 10$ follows by the same argument.
\end{proof}

Figure~\ref{fig:DBUsubs} shows the lattice of subalgebras of 
$\fourU^2$.  In the figure the indecomposable subalgebras are unshaded   
  and the decomposable ones  are shaded.

\begin{figure}  [ht]
\includegraphics[scale=1, trim=0  650  22 50]{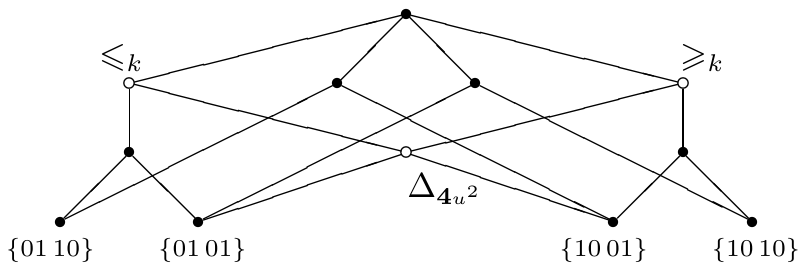}
\caption{The subalgebras of $\fourU^2$} \label{fig:DBUsubsnew}
\end{figure}

To list the piggybacking relations for  
$\DBU$ we first need to 
establish 
some notation. 
For $\w, \w_1,\w_2\in\Hu\Uu(\fourU)$ 
and $i\in\{0,1\}$, let  $R_{\w_1,\w_2}$  and $R_{\w}^i$ be
as 
defined in Theorem~\ref{genpigoneMu}. 
We write 
$r_{\w_1,\\w_2}$, respectively   $r_{\w}^i$, for the unique element of 
$R_{\w_1,\w_2}$, respectively $R_{\w}^i$,  whenever this set is a singleton, 
The set 
$\Hu\Uu(\fourU)$
 contains four elements: the maps 
$\alpha$ and $\beta$ defined earlier, and the  
constant maps onto $0$ and $1$, which we shall denote by
$\zerobar$ and~$\onebar$, respectively.
The following result is an easy consequence of
 Proposition~\ref{fourUsubs}.

\begin{prop} \label{DBUpigsub}   
Consider $\M=  \fourU$. Then 
\begin{newlist}
\item[{\rm (i)}]  for the cases in which  $R_{\w_1,\w_2}$
 is a singleton,  
\begin{newlist}
\item[{\rm (a)}] 
$r_{\alpha,\alpha} $ is $\leq_k$ and $r_{\beta,\beta} $ is $\geq_k$,
\item[{\rm (b)}] 
$r_{\w_1,\w_2} = \M^2$ whenever $\w_1 = \zerobar$ or $\w_2 =
\onebar$,
\item[{\rm (c)}] $r_{\alpha,\zerobar} = \{01\} \times \M$,
$r_{\beta,\zerobar} = \{10\} \times \M$,
$r_{\onebar,\alpha} = \M \times \{10\}$, and 
${r_{\onebar, \beta} = \M \times \{ 01\}}$;
\end{newlist} 
\item[{\rm (ii)}]  for the cases in which  $R_{\w_1,\w_2}$
 is not a singleton,  
\begin{newlist}
\item[{\rm (a)}] $R_{\alpha,\beta} = \bigl\{\{01\} \times \M, \{10\ 01\}\bigr\}$,
\item[{\rm (b)}] $R_{\beta,\alpha} = \bigl\{ \{10\}\times\M, \{10\ 01\}\bigr\}$,
\item[{\rm (c)}] $R_{\onebar,\zerobar}=\emptyset$;
\end{newlist}
\item[{\rm (iii)}] 
\begin{newlist}
\item[{\rm (a)}] 
$r_\alpha^0 = r_\beta^1 = \{ 01\}$ and     
$r_\alpha^1= r_\beta^0 =\{10\}$,
\item[{\rm (b)}] $r_{\zerobar}^0 = 
r_{\onebar}^{1} = \M$ and 
$R_{\zerobar}^1= R_{\onebar}^0 = \emptyset$.
\end{newlist} 
\end{newlist}   
\end{prop}
Below,
 when we describe the connections between the 
natural  
and
 Priestley-style dualities for $\DBU$, we shall see that the 
subalgebras listed in Proposition~\ref{DBUpigsub}
are exactly the relations we would expect to appear.   

We now present our duality theorem for $\DBU$.

\begin{thm} \label{DBUnatdual}   {\rm  (Natural duality for 
unbounded distributive bilattices)} 
There is a strong, and hence full, duality  between the category $\DBU$ and the category 
$\CP_{01}$ of  doubly-pointed Priestley spaces set up by hom-functors.  Specifically,
let 
\[
\fourU=  \bigl( \{ 00, 01, 10, 00\}; \lor_t, \land_t, \lor_k, \land_k, 
\neg\bigr)
\]
 be the four-element bilattice in the variety 
$\DBU$ of distributive bilattices without bounds
and let its alter ego be
\[
\fourDBTU = \bigl(  \{ 00, 11,  01, 10\};  
01, 10,
 \leq_k, \Tp\bigr).
\]
where the elements $01$ and $10$  are
 treated as  nullary operations.  
Then  
\[
\DBU = \ISP(\fourU) \quad \text{and} \quad \CP_{01} = 
\IScP (\fourDBTU)
\]
 and  the  hom-functors  
$\D = \DBU(-, \fourU)$ and 
$\E = \CP_{01}(-, \fourDBTU)$ set up the required dual equivalence 
between $\DBU$ and $\CP_{01}$.   
\end{thm}

\begin{proof}  Here we have included 
in the alter ego fewer relations than the  full set of piggybacking relations 
as listed in Proposition~\ref{DBUpigsub}
and we need to ensure that our restricted list suffices.  To accomplish 
this we use simple facts about entailment
as set out in \cite[Subsection~2.4.3]{CD98}.

We have included  as nullary operations 
both 
$01$ and $10$
and these entail 
the two one-element subalgebras 
$\{01\}$ and $\{10\}$ 
of    
$\fourU$.
It then follows from Theorem~\ref{genpigoneMu} and 
Proposition~\ref{DBUpigsub} that~$\fourDBTU$ yields a duality on 
$\DBU$ (see \cite[Section~2.4]{CD98}).  We now
invoke the $\MT$-Shift Strong Duality Lemma \cite[3.2.3]{CD98} to confirm that changing the alter ego by  removing entailed relations
 does not result in a duality which fails to be strong.

Finally, we note that $\fourDBTU$ is a doubly-pointed Priestley space and hence a member of $\CP_{01}$.  In the other direction,
$\twiddle 2_{01}$ is isomorphic to a closed substructure of $\fourDBTU$ and  
so belongs to 
$\IScP(\fourDBTU)$.  Hence the dual category for the natural duality is indeed the category of doubly-pointed Priestley spaces.  
\end{proof}

\section{How to dismount from a  
piggyback ride}\label{Sec:DisPiggyDual}

The piggyback method, applied to a  class $\CA=\ISP(\M)$ 
of $\CCD$-based algebras, 
supplies an alter ego $\MT$ yielding a natural duality for $\CA$,
 as described in Section~\ref{piggyonesorted}.
The relational structure of $\MT$  is constructed by 
bringing together ~$\twoT$ 
(the alter ego for Priestley duality for $\ISP(\two)$)  
and $\fnt{HU}(\M)$ (the Priestley dual space of 
the 
distributive lattice 
reduct 
 of the generating algebra  of $\CA$).
This  characteristic  of the piggyback method has 
a significant consequence:   it allows us, in a systematic way, 
to recover the Priestley dual spaces  $\fnt{HU}(\A)$ of the 
$\CCD$-reducts of the algebras  $\A\in\CA$.
The procedure for doing this played a central role in \cite{CPcop},
 where it was used to study coproducts
in quasivarieties of $\CCD$-based algebras.  
Below, in Theorem~\ref{Theo:RevEng},  we shall  strengthen  
Theorem~2.3  of \cite{CPcop}  by proving  that 
the construction given there 
is functorial and is 
naturally equivalent to $\fnt{HU}$.

Traditionally, dualities for $\CCD$-based (quasi)varieties 
have taken two  forms: natural dualities,
 almost always for classes $\CA$ which are finitely generated, 
and dualities which we dubbed ${\CCD\text{-}\CP\text{-based}}$
dualities in \cite[Section~2]{CPcop}. In the latter, at the object level,
the Priestley spaces of the $\CCD$-reducts
of members of $\CA$  are equipped with additional 
structure so that the operations of each algebra $\A$  in $\CA$ 
may be captured on $\fnt{KHU}(\A)$ (an isomorphic copy of 
$\fnt{U}(\A))$ from the structure imposed on the 
Priestley space $\fnt{HU}(\A)$.
Now assume that $\CA = \ISP(\M)$, where $\M$ is finite, 
so that a rival, natural, duality can be obtained by 
the piggyback method.
Reconciliations of the two approaches appear rather rarely in 
the literature; we can however draw attention to
\cite[Section~3]{DP87}  and the remarks in \cite[Section~7.4]{CD98}.  
There are two ways one might go in order to effect a reconciliation.
Firstly, we could use the fact that 
 an algebra $\A$ in $\CA$  determines and is determined by its natural dual $\D(\A)$
and  that $\U(\A)$ determines and is determined by $\fnt{HU}(\A)$. 
Given that, as we have indicated, we can
determine $\fnt{HU}(\A)$ from $\D(\A)$,
we could  try to capitalise on this  to  discover how to 
enrich the  Priestley  spaces  $\fnt{HU}(\A)$ to 
recapture the algebraic information lost in passage to the reducts. 
But this misses a key point about duality theory.  
The reason Priestley duality is such a useful tool is that 
it allows us concretely and in a functorial way to represent
distributive lattices 
in terms of Priestley spaces.  
Up to categorical isomorphism,  
it is immaterial how the dual spaces are actually constructed.  
An alternative strategy  now suggests itself for obtaining a duality 
for~$\CA$ based on  enriched Priestley spaces.

What we shall do in this section  is to work with a version of
 Priestley duality based on structures directly derived from 
the natural duals $\D(\A)$ of the algebras $\A$, 
rather than one based on  traditional Priestley duality applied to 
the class $\fnt{U}(\CA)$.
This shift of viewpoint  allows us to tap in to  the information
 encoded in the natural duality in a rather transparent way.  
We can  hope thereby to  arrive  at a `Priestley-style'  duality 
for $\CA=\ISP(\M)$. 
We shall demonstrate how this can be carried out in cases 
where the operations suppressed by the forgetful functor interact 
in a particularly well-behaved way with the operations which 
are retained.    
At the end of the section we also record how the strategy 
extends  to $\DU$-based algebras. 

In summary, we propose to base Priestley-style dualities on 
dual categories more closely linked to natural dualities rather than,
 as in the literature, seeking to enrich Priestley duality per se. 
The two approaches are essentially equivalent, but ours 
has several  benefits.  
By staying close to a natural duality  we are well placed to profit 
from the good categorical properties such a duality possesses. 
Moreover morphisms are treated alongside objects.
 Also, setting up a piggyback duality is  an algorithmic process 
in a way that formulating a Priestley-style duality ab initio is not. 
Although we restrict attention in this paper 
to the special types of operation present in bilattice varieties, 
and these could be handled by more traditional means, 
we note that our analysis has the potential to be adapted  
to other situations.

We now recall the construction of \cite[Section~2]{CPcop} as it 
applies 
 to the particular case of the piggyback theorem for the bounded case as stated in Theorem~\ref{genpigoneM}. 
Assume that $\M$ and  $R$ 
are as in that theorem. 
For a fixed algebra $\A \in \ISP(\M)$, we 
define 
$  Y_{\A}= \D(\A)\times\Omega,$
where $\Omega=\CCD(\U (\A),\two)$,
and equip it  
with the topology
$\Tp_{Y_{\A}}$
 having as a base of open sets
\[  \Tp_{Y_{\A}}=  \{\,U\times V\mid  U\mbox{ open in } \D(\A)\mbox{ and }V\subseteq\Omega\,\}
\]
and  with 
 the binary relation  $\preccurlyeq\, \,\subseteq Y_{\A}^2$ defined by 
\[
    (x,\w_1)\preccurlyeq(y,\w_2)\mbox{ if }(x,y)\in r^{\D(\A)} \mbox{ for some }r\in R_{\w_1,\w_2}.  
\]

In \cite[Theorem~2.3]{CPcop}, we proved that the binary relation $\preccurlyeq$ 
 is a pre-order on 
$Y_{\A}$. Moreover, if $\approx\,=\,\preccurlyeq\cap\succcurlyeq$  
				    denotes the equivalence relation on 
$Y_{\A}$ determined by $\preccurlyeq$
and 
 $\Tp_{Y_{\A}}/_{\approx}$ 
is the quotient topology, then 
$  (\, Y_{\A}/_{\approx};{\preccurlyeq}/_{\approx},\Tp_{Y_{\A}}/_{\approx}) $
  is a  Priestley space isomorphic to $\fnt{HU}(\A)$. This isomorphism is determined by the map 
$\Phi_{\A}$ given by 
$\Phi_{\A}([(x,\w)]_{\approx})=\w\circ x$.

\begin{thm}\label{Theo:RevEng}  Let $\CA = \ISP(\M)$, where 
$\M$ is a finite algebra with
a reduct in $\CCD$.
Then there exists a well-defined 
contravariant functor $\fnt{L}\colon\CA\to\CP$ given~by
\begin{alignat*}{3}
&\text{on objects:} & \hspace*{1.1cm} &
\A \longmapsto\  \fnt{L}(\A) =  (\, Y_{\A}/_{\approx};{\preccurlyeq}/_{\approx},\Tp_{Y_{\A}}/_{\approx}), \hspace*{1.1cm} \phantom{\text{on objects:}}&&\\  
&\text{on morphisms:}  &   & \, \,h 
\longmapsto\ \fnt{L}(h) \colon 
[(x,\w)]_{\approx} \mapsto  [(\D(h)(x),\w)]_{\approx}. &&
\end{alignat*} 
Moreover, $\Phi$, defined on each $\A$ by 
$\Phi_{\A} \colon [(x,\w)]_{\approx} \mapsto \w\circ x$,
determines a natural isomorphism between $\fnt{L}$  and $\fnt{HU}$.
\end{thm} 
\begin{proof}
We have already noted 
 that $\fnt{L} ( \A)\in\CP $.
We 
confirm that 
$\fnt{L}$ is a functor.  Let $h\colon\A\to \B$ and 
$(x,\w),(y,\w')\in Y_{\B}$  be
 such that $(x,\w)
\preccurlyeq  
(y,\w')$. 
Then there 
exists
 $r\in R_{\w,\w'}$  with 
$(x,y)\in r^{\D(\B)}$. Hence
${(\D(h)(x),\D(h)(y))\in r^{\D(\A)}}$
and $(\D(h)(x),\w)\preccurlyeq 
(\D(h)(y),\w')$.  
Thus 
$\fnt{L}(h)$ is
well defined and
 order-preserv\-ing.
Since $\D(h)$ is continuous and
$Y_{\A}/_{\approx}$ carries the quotient topology, and since
$\fnt{L}(h)^{-1}(U\times V)=\D(h)^{-1}(U)\times V$, 
the map  
$\fnt{L}(h)$ is also continuous.

Theorem~3.1(c) in \cite{CPcop}  proves that 
$\Phi_{\A}\colon \fnt{L}(\A)\to \fnt{HU}(\A)$ is an isomorphism of
 Priestley spaces. 
We prove that $\Phi$ is natural in $\CA$.  
Let $\A,\B\in\CA$, $x\in\D(\B)$, $h\in\CA(\A,\B)$ and $\w\in\Omega$. Then
\begin{multline*} 
\Phi_{\A}(\fnt{L}(h)([(x,\w)]_{\approx}))=\Phi_{\A}([(\fnt{D}(h)(x),\w)]_{\approx})=\Phi_{\A}([(x\circ h,\w)])\\
=\w\circ x\circ h=\fnt{H}(h)(\w\circ x)=\fnt{HU}(h)(\w\circ x) 
=\fnt{HU}(h)(\Phi_{\B}([(x,\w)]_{\approx})).
\end{multline*}
Therefore
$\Phi$ is a natural isomorphism between  the functors $\fnt{L}$ and~$\fnt{HU}$.
\end{proof}

We take  as before a $\CCD$-based quasivariety $\CA=\ISP(\M)$,
 with forgetful functor 
$\U \colon \CA \to \CCD$, 
for which we have set up a piggyback duality.
Theorem~\ref{Theo:RevEng} 
tells us how, given an algebra $\A\in\CA$,  to obtain 
from the natural 
dual $\D(\A)$
a Priestley space 
$Y_{\A}/_{\approx}$ 
serving as 
the 
dual  space of 
$\fnt{U}(\A)$.
But it  does  not yet  
tell us  how to capture on 
$Y_{\A}/_{\approx}$ 
the algebraic operations not present in the reducts.  
However it should be  borne in mind that the  
maps $\w$ in 
$\Omega=\fnt{HU}(\M)$ 
are an integral part of the natural duality construction
and it is therefore unsurprising that these maps will play a direct role in the translation
to a Priestley-style duality, if we can achieve this.  
We consider in turn operations of each  of the types 
present in the bilattice context.

Assume  first 
that 
$f$ is a unary operation occurring in the type of algebras in $\CA$   which interprets as a  $\CCD$-endomorphism on each $\A \in \CA$.
Then $\fnt{H}(f^{\A})\colon\fnt{HU}(\A)\to\fnt{HU}(\A)$ is a continuous order-preserving
map,  
given 
by $\fnt{H}(f^{\A})(x)=x\circ f^{\A}$,  for each $x\in
\fnt{HU}(\A)$. Conversely,~$f^{\A}$ can be recovered from $\fnt{H}(f^{\A})$ 
by setting $f^{\A}(a)$ for each $a\in\A$
to be  
the unique element of $\A$
for which  
 $x(f^{\A}(a))=(\fnt{H}(f^{\A})\circ x)(a)$ for each
 $x\in\fnt{HU}(\A)$.  Denote $\fnt{H}(f^{\A})$ by~$\widehat{f^{\A}}$.

Then for each $\A\in\CA$ the operation 
 $f^{\A}$ is determined by $f^{\M}$. 
Dually, $\widehat{f^{\M}}$ should encode enough information to
 enable us, with the aid of Theorem~\ref{Theo:RevEng}, to 
recover $\widehat{f^{\A}}$. 
Define a map $f_{Y_{\A}} \colon Y_{\A} \to Y_{\A}$ by 
${f_{Y_\A}(x,\w) =  (x,\w \circ f^{\M})}$, 
for  $x \in \D(\A)$ and  $\w \in \Omega$; 
here $Y_{\A}=\fnt{D}(\A)\times \Omega$, 
as in Theorem~\ref{Theo:RevEng}. 
By definition of $(Y_{\A}; \preccurlyeq,\Tp_{\A})$, 
the map $f_{Y_\A}$ is continuous.
By Theorem~\ref{Theo:RevEng}(c),
 for  every $x,x' \in \D(\A)$ and  $\w,\w' \in \Omega$, 
\begin{align*}
(x,\w)\approx (x',\w')&\Longleftrightarrow \w\circ x=\w'\circ x' \\
&\Longrightarrow \w\circ f^{\M}\circ x= \w\circ x\circ f^{\A}=\w'\circ x'\circ f^{\A}=\w'\circ f^{\M}\circ x' \\
&\Longleftrightarrow f_{Y_\A}(x,\w)\approx f_{Y_\A}(x',\w'). 
\end{align*}
Then
 the map
$\bar{f}_{\A}\colon Y_{\A}/_{\approx}\to Y_{\A}/_{\approx}$ 
determined by  
$\bar{f}_{\A}([(x,\w)]_{\approx})=[f_{Y_{\A}}(x,\w)]_{\approx}$
 is  well defined and continuous.
For each  $(x,\w) \in Y_{\A}$ and $a\in\A$ we have
\begin{multline*}
\widehat{f^{\A}}
(\Phi_{\A}([(x,\w)]_{\approx}))
(a)=\w\circ x(f^{\A}(a)) 
= \w (f^{\M}(x(a)) )  
= (\w \circ f^{\M})(x(a))  \\
=\Phi_{\A} ([(x,\w \circ f^{\M})])(a) 
 = \Phi_{\A}(\bar{f}_{\A}([(x,\w)]))(a).
\end{multline*}
We have proved that 
$\widehat{f^{\A}}\circ\Phi_{\A}=\Phi_{\A}\circ \bar{f}_{\A}$.

We now consider a unary  operation  $h$ which interprets  as a dual $\CCD$-endo\-morph\-ism on each~$\U(\A)$.  
As above,  
 $\fnt{H}(h^{\A})\colon\fnt{HU}(\A)\to\fnt{HU}(\A^{\partial})$ is a continuous order-preserving map.
Using the fact that the assignment $x\mapsto \cnst{1}-x$
defines an isomorphism between  the 
Priestley 
spaces  $\fnt{HU}(\A)^{\partial}$ and  $\fnt{HU}(\A^{\partial})$, it is possible to  define a map $\widehat{h^{\A}}\colon\fnt{HU}(\A)\to\fnt{HU}(\A)$ by $\widehat{h^{\A}}(x)
=\cnst{1}-\fnt{H}(h^{\A})(x)=\cnst{1}-(x\circ h^{\A})$. Then $\widehat{h^{\A}}$ is continuous and order-reversing.
 Conversely, $h^{\A}$ is obtained from $\widehat{h^{\A}}$ 
by setting $h^{\A}(a)$ to be the unique element of $\A$ 
that satisfies $x(h^{\A}(a))=(\cnst{1}-(\widehat{h^{\A}}(x)))(a)$ 
for each $x\in\fnt{HU}(\A)$. 
In the same way as before, we define a map 
$h_{Y_\A} \colon Y_{\A} \to Y_{\A}$ given by 
$h_{Y_\A}(x,\w) =  (x,\cnst{1}-\w \circ h^{\M})$. 
Again we have an associated continuous (now order-reversing) map 
on $(Y_{\A}; \preccurlyeq,\Tp_{\A})$ given by 
\[\bar{h}_{\A}([(x,\w)]_{\approx}) = [h_{Y_\A}(x,\w) ]_{\approx} = [(x,\cnst{1}-\w \circ h^{\M})]_{\approx}.\] 
Furthermore, 
 $\widehat{h^{\A}}\circ\Phi_{\A}=\Phi_{\A}\circ \bar{h}_{\A}$.

Nullary operations are equally simple to handle.  
Suppose the algebras in $\CA$ contain a nullary operation $\cnst c$ in the type.  
Then for each $\A\in\CA$ the 
constant~$\cnst c^{\A}$ determines a clopen up-set
 $\widehat{\cnst c^{\A}}=
\{\,x\in\fnt{HU}(\A)\mid x(\cnst c^{\A})=1\,\}$  in  
$\fnt{HU}(\A)$.
Conversely, $\cnst c^{\A}$  is the unique element $a$ of $\A$ such that $x(a)=1$ if and only if 
$x\in \widehat{ \cnst c^{\A}}$. 
Now let $\cnst c_{Y_\A}=\D(\A)\times \{\,\w\in\Omega\mid \w(\cnst c^{\M})=1\,\}$. 
In the same way as above we can move down to the Priestley space level 
and define 
\[
\bar{\cnst c}_{\A} 
=\{\, [(x,\w)]_{\approx} \mid (x,\w) \in  \cnst c_{Y_\A}\,\}=\{\, [(x,\w)]_{\approx} \mid \w( \cnst c^{\M})=1\,\}.
\]
Then, for each $(x,\w) \in Y_{\A}$, we have 
\begin{multline*} 
\Phi_{\A}([(x,\w)]_{\approx})\in\widehat{\cnst c^{\A} }\Longleftrightarrow 1=(\w\circ x)
 (\cnst c^{\A})=\w(\cnst c^{\M})\\
 \Longleftrightarrow
(x,\w)\in\cnst c_{Y_\A}
 \Longleftrightarrow
[(x,\w)]_{\approx}\in \bar{\cnst c}_{\A}.
\end{multline*} 
That is, $\Phi_{\A}$ and its inverse interchange  the sets $\bar{\cnst c}^{\A}$ and $\widehat{\cnst c^{\A}}$.

We sum up in the following theorem  what we have shown on how enriched Priestley spaces may be obtained which encode the non-lattice operations of an algebra $\A$ with a reduct $\U(\A)$ in $\CCD$.
Following common practice in similar situations,  we shall 
simplify the presentation by assuming that only one operation of each kind is present.   
To 
state 
 the theorem we need
a definition.
 Let $\CY$ be the category whose objects are  the structures of the form 
$(\Y; p,q,S)$, where~$\Y$ is a Priestley space, $p$ and~$q$ are continuous self-maps on~$\Y$
which are respectively order-preserving and order-reversing, and $S$
is a distinguished clopen subset of~$\Y$. The morphisms of~$\CY$ are continuous order-preserving maps that commute with $p$ and $q$, and preserve $S$.

\begin{thm}  \label{Theo:RevEngOps}
Let $\CA =\ISP(\M)$  be a finitely generated quasivariety
for which the language is that of $\CCD$ augmented with 
two unary operation symbols, $f$ and~$h$, 
 and a nullary operation symbol~$c$
such that, for each $\A \in \CA$,
  \begin{newlist}  
\item[{\rm (i)}]  $f^{\A}$ acts as an endomorphism  of $\CCD$, and
\item[{\rm (ii)}]  $h^{\A}$ acts  as  a dual
 endomorphism  of $\CCD$.
\end{newlist}
 Then there exist well-defined contravariant  functors $\fnt{L}^+$ and $\fnt{HU}^+$ 
from $\CA$ to~$\CY$ given by
\begin{alignat*}{3}
&\text{on objects:} & \hspace*{1.72cm} & \fnt{L}^+ \colon  \A \mapsto  (\fnt{L}(\A), \bar{f}_{\A}, \bar{h}_{\A},   
\bar{\cnst c}_{\A}),   \hspace*{1.72cm} 
\phantom{\text{on objects:}} &&\\
&\text{on morphisms:}  &  &\fnt{L}^+\colon \,\, h \mapsto\fnt{L}(h);&& \\
\shortintertext{and} 
&\text{on objects:} & &\fnt{HU}^+ \colon  \A \mapsto  (\fnt{HU}(\A);   \widehat{f^{\A}}, \widehat{h^{\A}},   
\widehat{\cnst c^{\A}}),&&\\ 
&\text{on morphisms:}  & &\fnt{HU}^+\colon \,\, h \mapsto\fnt{HU}(h).&&
\end{alignat*}
 Moreover,
$\Phi$, as defined in Theorem~{\upshape\ref{Theo:RevEng}},  is a natural equivalence between the functor $\fnt{L}^+$ and the functor $\fnt{HU}^+$. 

Let 
$\CY'$ denote the full subcategory of $\CY$ whose objects are isomorphic to 
topological structures of the form
$\fnt{L}^+(\A)$ {\rm(or }equivalently~$\fnt{HU}^+(\A)${\rm)} for some $\A\in\CA$.
the categories $\CA$ and $\CY$ are dually equivalent, with  
 the equivalence determined by either $\fnt{L}^+$ or $\fnt{HU}^+$.
\end{thm}

We now indicate the 
 modifications that we have to 
make 
 to 
Theorem~\ref{Theo:RevEng} to handle the 
unbounded case. 
In Theorem~\ref{Theo:RevEngu},  the  sets of relations arising  are as specified   in 
Theorem~\ref{genpigoneMu}.

 Let $\CA = \ISP(\M)$, where 
$\M$ is a finite algebra having  a reduct $\Uu(\M)$  in $\DU$ and let $\Omega=\Hu\Uu(\M)$. 
For each $\A\in \CA$, let 
$
  \textstyle  Y_{\A}= \Du(\A)\times\Omega
$
with the topology $\Tp_{Y}$ having as 
a base  of open sets  
$
    \{\,U\times V\mid  U\mbox{ open in } \Du(\A)\mbox{ and }V\subseteq\Omega\,\},
$
and the binary relation $\preccurlyeq\, \subseteq Y^2$ given by
\[
    (x,\w_1)\preccurlyeq(y,\w_2)\mbox{ if }(x,y)\in r^{\Du(\A)}\mbox{ for some }r\in R_{\w_1,\w_2}.
\]

\begin{thm}\label{Theo:RevEngu}  Let $\CA = \ISP(\M)$, where 
$\M$ is a finite algebra with
 a reduct in $\DU$.
Then there exists a well-defined contravariant  functor $\unbounded{\fnt{L}}\colon\CA\to\CP_{01}$ given by
\begin{alignat*}{3}
&\text{on objects:} & \hspace*{.55cm} &
 \A \longmapsto\   \unbounded{\fnt{L}} (\A) = (\, Y_{\A}/_{\approx}; {\preccurlyeq}/_{\approx}, c_0, c_1, \Tp_{Y_{\A}}/_{\approx}),\hspace*{.55cm}
\phantom{\text{on objects:}}
 && \\
&\text{on morphisms:}  & &
   h \,\,
\longmapsto\ 
\unbounded{\fnt{L}}(h)\colon [(x,\w)]_{\approx}
\mapsto 
[(\Du(h)(x),\w)]_{\approx}.  && 
\end{alignat*}

Moreover, $\Phi$, defined on each $\A$  by $\Phi_{\A}([(x,\w)]_{\approx})=\w\circ x$, determines a natural isomorphism between $\unbounded{\fnt{L}}$  and $\Hu\Uu$.
\end{thm}

\begin{proof}
The only new ingredient here as compared with the proof of Theorem~\ref{Theo:RevEng} concerns the role of the constants. 
The argument used in the proof of that theorem, as given in~\cite[Theorem~2.3]{CPcop}, can be applied directly to prove that $\Phi_{\A}\colon(\, Y_{\A}/_{\approx}; {\preccurlyeq}/_{\approx},
\Tp_{Y_{\A}}/_{\approx})\to \Hu\Uu(\A)$ defined by 
$\Phi_{\A}([(x,\w)]_{\approx})=\w\circ x$ is a well-defined homeomorphism which is also an order-isomorphism.
To confirm  that $\unbounded{\fnt{L}}$ is well defined 
we shall show  simultaneously 
that $\left(\bigcup \{\,R^{i}_{\w}\mid \w\in\Omega\,\}\right)/_{\approx}$ is a singleton and that
$\Phi_{\A}$ maps 
its
 unique element  
to the corresponding constant map 
in  $\Hu\Uu(\A)$.
Thus 
$\textstyle 
\{c_i\}=\left(\bigcup \{\,R^{i}_{\w}\mid \w\in\Omega\,\}\right)/_{\approx}
$
for
$i\in\{0,1\}$.

 Below we write~$r$ rather than 
$r^{\Du(\A)}$ for the lifting of a piggybacking relation~$r$ to $\Du(\A)$.
Let $\w_1,\w_2\in\Omega$ and $r_1\in R^{1}_{\w_1}$, $r_2\in R^{1}_{\w_2}$, $x\in r_1$, and $y\in r_2$. 
For each $a\in \A$, we have  $\w_1 (x(a))=1=\w_2 (y(a))$. Then $\Phi_{\A}([
(x,\w_1)]_{\approx})=\Phi_{\A}([(x,\w_2)]_{\approx})=\cnst 1$, where $\cnst 1\colon A\to \{0,1\}$ denotes the constant map $a\mapsto 1$ . Since~$\Phi_{\A}$ is injective, $[(x,\w_1)]_{\approx}=\nu([(x,\w_2)]_{\approx})$. 
This proves that $|\bigcup \{\, R^1_{\w}\mid \w\in\Omega\,\}/_{\approx}|\leq 1$ and that $\Phi_{\A}((\bigcup \{R^{1}_{\w}\mid \w\in\Omega\})/_{\approx})\subseteq\{\cnst 1\}$.
Similarly, we obtain $|\bigcup \{R^0_{\w}\mid \w\in\Omega\}/_{\approx}|\leq 1$ and $\Phi_{\A}(\bigcup \{R^{0}_{\w}\mid \w\in\Omega\})/_{\approx})\subseteq\{\cnst 0\}$. 
Because 
~$\Phi_{\A}$ is surjective,  
there exists $x\in \Du(\A)$ and $\w\in\Omega$ such that $\w \circ x 
=\cnst 1$. Then $x\in R^1_{\w}$, which proves that $\bigcup \{R^1_{\w}\mid \w\in\Omega\}\neq\emptyset$. The same argument applies to $\bigcup \{R^0_{\w}\mid \w\in\Omega\}$. 
\end{proof}

The arguments 
for handling additional operations in the bounded case 
carry over  to 
 piggyback dualities over~$\DU$ with only the obvious modifications.

\section{From a natural duality to the product representation} \label{sec:prodrep}

The natural dualities in Theorems~\ref{DBnatdual} and~\ref{DBUnatdual} combined with the Priestley dualities for bounded and unbounded distributive lattices, respectively, prove that $\DB$ is categorically equivalent to $\CCD$ and that $\DBU$ is categorically equivalent to $\DU$. These equivalences are
set up
 by the functors $\fnt{KD}\colon \DB\to\CCD$ and  $\fnt{EH}\colon \CCD\to\DB$, and $\Ku\Du\colon \DBU\to\DU$ and  $\Eu\Hu\colon \DU\to\DBU$:

\begin{center}
\begin{tikzpicture} 
[auto,
 text depth=0.25ex,
 move up/.style=   {transform canvas={yshift=1.9pt}},
 move down/.style= {transform canvas={yshift=-1.9pt}},
 move left/.style= {transform canvas={xshift=-2.5pt}},
 move right/.style={transform canvas={xshift=2.5pt}}] 
\matrix[row sep= 1cm, column sep= 1.3cm]  
{ 
\node (DB) {$\DB$}; & \node (P) {$\CP$};  & \node (D) {$\CCD$};&\node (DBU) {$\DBU$}; & \node (PU) {$\CP_{01}$};  & \node (DU) {$\DU$.};\\ 
};
\draw  [->, move up] (DB) to node  [yshift=-2pt] {$\D$}(P);
\draw [<-,move down] (DB) to node [swap]  {$\E$}(P);
\draw  [->,move up] (P) to node  [yshift=-2pt] {$\fnt{K}$}(D);
\draw [<-,move down] (P) to node [swap]  {$\fnt{H}$}(D);
\draw  [->, move up] (DBU) to node  [yshift=-2pt] {$\Du$}(PU);
\draw [<-,move down] (DBU) to node [swap]  {$\Eu$}(PU);
\draw  [->,move up] (PU) to node  [yshift=-2pt] {$\Ku$}(DU);
\draw [<-,move down] (PU) to node [swap]  {$\Hu$}(DU);
\end{tikzpicture}
\end{center}
With the aid of Theorem~\ref{Theo:RevEng}
we can
give  explicit 
descriptions   of  
$\fnt{EH}$ and~$\fnt{KD}$.

\begin{thm} \label{Theo:RevEngDB}  Let 
$\D\colon\DB\to\CP$ and $\E\colon\CP\to\DB $ be the functors setting up the duality presented in Theorem~{\upshape\ref{DBnatdual}}.
 Then for each $\A\in \DB$ the Priestley dual 
$\fnt H(\A_t)$ of the $t$-lattice reduct of $\A$ is 
such that  
\[
\textstyle\fnt H(\A_t)\cong\D(\A)\coprod_{\CP}\D(\A)^\partial,
\]
where $\cong$ denotes an isomorphism of Priestley spaces. 
\end{thm}

\begin{proof}
Adopting  the notation of Theorems~\ref{genpigoneM} and \ref{DBnatdual}, 
we note that in the 
proof of the latter we observed that
\[
R_{\alpha,\beta} = R_{\beta,\alpha}=\emptyset, \qquad  
r_{\alpha,\alpha}\mbox{ is } \leq_{k}  \ \text{ and }\   r_{\beta,\beta}\mbox{ is } \geq_{k}
\]
(here we have written 
 $r_{\w,\w}$ for  the unique element of $R_{\w,\w}$).
 As a result, for  $\A\in\DB$, with $\D(\A)=(X;\leq,\Tp)$, we have 
\[
R^{\D(\A)}_{\alpha,\beta} = R^{\D(\A)}_{\beta,\alpha}=\emptyset,  \qquad 
r^{\D(\A)}_{\alpha,\alpha}\mbox{ is } \leq \ \text{ and } \ r^{\D(\A)}_{\beta,\beta}\mbox{ is } \, \geq.
\]
From this and the definition of $\preccurlyeq\, \subseteq Y_{\A}^2$ it follows that
\[(x,\w_1)\preccurlyeq (y,\w_2) 
\Longleftrightarrow 
\begin{cases}
 x\leq y  \mbox{ and } \w_1=\w_2=\alpha, \mbox{ or }\\
  x\geq y  \mbox{ and } \w_1=\w_2=\beta.
\end{cases}
\]
Then $Y_{\A}=(\D(\A)\times\Omega;\preccurlyeq,\Tp_{Y_{\A}})$ 
is already a poset (no quotienting is required) for each $\A\in\DB$. 
 And, order theoretically and topologically,  $Y_{\A}$ is the 
disjoint union of ordered spaces $Y_\alpha$ and~$Y_\beta$, 
where $Y_\alpha$ and  $Y_\beta$ are the subspaces of $Y_{\A}$
 determined by $\D(\A)\times\{\alpha\}$ and
 $\D(\A)\times\{\beta\}$, respectively.
 With this notation we also have $Y_{\alpha}\cong\D(\A)$ and
 $Y_\beta\cong\D(\A)^\partial$.  
The rest of the proof follows directly from
 Theorem~\ref{Theo:RevEng} and the fact that finite coproducts 
in $\CP$ correspond to disjoint unions \cite[Theorem 6.2.4]{CD98}.
\end{proof}

\begin{figure}  [ht]
\begin{center}
\begin{tikzpicture} [scale=.55]
 [inner sep=10mm,
 [auto,
 text depth=0.25ex,
 wiggly/.style={decorate,decoration={snake,amplitude=1pt,segment length=5pt,
                pre length=3pt,post length=3pt}}, 
]

\draw [thin] (-4.5,-2)--(-2.5,-2) --(-2.5,1)--(-4.5,1)--(-4.5,-2);
\draw  [thin] (-1,-2)--(1,-2) --(1,1)--(-1,1)--(-1,-2);

\draw [thin] (10,-2)--(12,-2) --(12,1)--(10,1)--(10,-2);
\draw  [thin] (12,-2)--(14,-2) --(14,1)--(12,1)--(12,-2);
\node (YA) at (-1.5,-3) {$(Y_{\A};\preccurlyeq)$};
\node (HUA) at (12.5,-3) {$\fnt{HU}(\A)$};
\node (Yalpha) at  (-5.2,0) {$Y_{\alpha}$};
\node (Ybeta) at (-1.6,0) {$Y_{\beta}$};
\node (geqk) at (0,-1) {$\geq_k$};
\node (leqk) at (-3.5,-1) {$\leq_k$};

\draw [->,decorate,decoration={snake,amplitude=1pt,segment length=5pt,
                pre length=3pt,post length=3pt}]
(3.5,-1)-- node [yshift=.5cm]  {$z \mapsto [z]_{\approx}$} (8,-1);  
\end{tikzpicture}
\end{center}
\caption{Obtaining  $\fnt{HU}(\A)$  from $\D(\A) $ \label{fig:bddquot}}
\end{figure}

Figure~\ref{fig:bddquot} shows the very simple way in which Theorem~\ref{Theo:RevEngDB}  tells us how to pass 
from the natural dual $\D(\A)$ of  $\A \in 
\DB$ 
to the Priestley space ${\fnt{HU}(\A)= \fnt{H}(\A_t)}$. 
We start from
 copies~$Y_\alpha$ and $Y_\beta$ 
of $\D(\A)$, indexed by the points $\alpha$ and $\beta$ of $\Omega
=\fnt{HU}(\A)$. 
 The relation $\preccurlyeq$ gives us the partial order on 
$Y_\alpha \cup Y_\beta$ which restricts to $\leq_k$ on $Y_\alpha$
and $\geq_k$ on $Y_\beta$.  The relation $\approx$ makes no identifications; in the right-hand diagram the two order comments 
are regarded as subsets of a single Priestley space; in the left-hand diagram 
they are regarded as two copies of the natural dual space.  This very simple picture should be contrasted with the somewhat more complicated one we obtain below for the unbounded case; see   Figure~\ref{fig:unbddquot}.

Theorem~\ref{Theo:RevEngDB} shows us
how to obtain $\fnt H(\A_t)$ from $\D(\A)$. 
We conclude that for each $\A\in\CA$, the $t$-lattice reduct of $\A$ is isomorphic to $\Lalg\times \Lalg^{\partial}$ where $\Lalg=\fnt{KD}(\A)$. We will now see how to 
capture 
in  
 $\fnt H(\A_t)$ 
the algebraic operations suppressed by $\U$. 
Drawing on Theorem~\ref{Theo:RevEngOps} 
we have
\begin{alignat*}{2}
\bar{\neg}_{\A}([(x,\alpha)])&=[(x,\beta)],
 &\bar{\neg}_{\A}([(x,\beta)])&=[(x,\alpha)];\\
\widehat{\neg^{\A}}(\alpha\circ x)&=\beta\circ x, 
 &\widehat{\neg^{\A}}(\beta\circ x) &=\alpha\circ x;\\
\bar{1_k}\,_{\A}&=Y_{\alpha}, &
\bar{0_k}\,_{\A}&=Y_{\beta};
\\
\widehat{1_k^{\A}}& =\{\,\alpha\circ x\mid x\in\D(\A)\,\},
&\widehat{0_k^{\A}}&=\{\,\beta\circ x\mid x\in\D(\A)\,\}.
\end{alignat*}
From this and Theorem~\ref{Theo:RevEngDB}, we obtain   ${\fnt{KD}(\A)\cong\A_t/\theta}$ for each $\A\in\DB$, where $\theta$ 
is the congruence defined by $a\, \theta\,  b$ 
if and only if $a\land_t 1_k=b\land_t 1_k$. 
Clearly~$\A_t/\theta$ is also 
isomorphic  
to the sublattice of~$\A_t$ determined by the set $\{\,a\in A\mid a\leq_t 1_k\,\}$.

Since the duality  we developed for $\DB$ was based on 
the piggyback duality using $\A_t$ as the $\CCD$-reduct,
 Theorem~\ref{Theo:RevEng} does not give  us direct access to 
the $k$-lattice operations.
 Lemma~\ref{90deg} tells us that with the knowledge constants
 and the $t$-lattice operations we can  access 
the $k$-lattice operations. 
But  there is a way to recover the $k$-lattice operations
 directly from the dual space,  and this  can be adapted to cover
 the unbounded case too.

Take,   as before,  $\CA =\DB$, $\M=\four$  and 
 $\Omega=\{\alpha,\beta\}$.
Let $\A\in\CA$ and $ Y_{\A}=\D(\A)\times\Omega$. 
Define a partial order $\preccurlyeq'\, \subseteq Y_{\A}^2$ by  
$(x,\w)\preccurlyeq'(y,\w')$ if $\w=\w'$ and $x\leq y$ in $\D(\A)$. 
It is clear that $(Y_{\A};\preccurlyeq',\Tp_{Y_{\A}})\cong \D(\A)\coprod_{\CP}\D(\A)$. 
We claim that $\fnt{H}(\A_k)\cong (Y_{\A};\preccurlyeq',\Tp_{Y_{\A}})$. 
To prove this,
observe that, since $\alpha^{-1}(1)=\{11,01\}$ is a filter of 
the lattice $\four_k$,  the map $\alpha$ is a lattice homomorphism 
from $\four_k$ into~$\two$. And since $\beta^{-1}(1)=\{11,10\}$ 
is an ideal in $\four_k$  the map $\beta'=\cnst 1  -\beta$,
is a lattice homomorphism from $\four_k$ into $\two$.  
It follows  that we have a well-defined  map 
$\eta_{\A}\colon Y_{\A}\to \fnt{H}(\A_k)$ given by
\[
  \eta_{\A}(x,\w)=\begin{cases}
    \w\circ x&\mbox{if } \w=\alpha,\\
    \cnst 1-\w\circ x&\mbox{if } \w=\beta.\\
  \end{cases}
\]

Assume that $(x,\w)\preccurlyeq'(y,\w')$. 
Then $\w=\w'$ and for each $a\in\A$ we have $x(a)\leq_k y(a)$ in $\four$. 
Since $\alpha$ is a $k$-lattice homomorphism, if $\w=\w'=\alpha$, then
 \[\eta_{\A}(x,\alpha)(a)=\alpha(x(a))\leq \alpha(y(a))=\eta_{\A}(y,\alpha)(a),\] for each $a\in\A$.
If instead $\w=\w'=\beta$, we have $\beta_{\A}(x(a))\geq \beta_{\A}(y(a))$ for each $a\in \A$, then $\eta_{\A}(x,\beta)(a)=1-\beta(x(a))\leq 1-\beta_{\A}(y(a))=\eta_{\A}(y,\beta)(a)$. 
 Therefore $\eta_{\A}$ preserves $\preccurlyeq'$. 
To see that~$\eta_{\A}$ also reverses the order,  assume $\eta_{\A}(x,\w)\leq \eta_{\A}(y,\w')$. Then  $\eta_{\A}(x,\w)(a)\leq \eta_{\A}(y,\w')(a)$ in $\two$, for each $a\in\A$. 
Since $\alpha(1_t)=1\not\leq 0=1-\beta(1_t)$ and $1=\beta(0_t)=1\not\leq 0=\alpha(1_t)$ it follows that $\w=\w'$. 
Now assume that $\w=\w'=\alpha$, then $\alpha(x(a))\leq \alpha(y(a))$, for each $a\in \A$, equivalently $(x(a),y(a))\in r_{\alpha,\alpha}=\leq_k$ for each $a\in\A$. 
By Theorem~\ref{DBUnatdual}, 
$x\leq y$ in $\D(\A)$. We obtain $(x,\w)\preccurlyeq'(y,w)$.
If $\w=\w'=\beta$ 
we argue in the same way,
 using the fact that  $r_{\beta,\beta}$ is $\geq_k$.

Finally, observe that for each $a\in\A$, $b\in\four$ and $i\in\two$,
\allowdisplaybreaks
\begin{align*}  
\eta_{\A}(\{\,x\in\D(\A)\mid x(a)=b\}\times\{\alpha\,\})&=\{\,z\in\fnt{H}(\A_k)\mid z(a)=\alpha(b)\,\}\\
&\ \ \cap\{\,z\in\fnt{H}(\A_k)\mid z(\neg^{\A}a)=\alpha(\neg^{\four} b)\,\};\\
\eta_{\A}(\{\,x\in\D(\A)\mid x(a)=b\,\}\times\{\beta\})&=\{\,z\in\fnt{H}(\A_k)\mid z(a)\neq\beta(b)\,\}\\
&\ \ \cap\{\,z\in\fnt{H}(\A_k)\mid z(\neg^{\A}a)\neq\beta(\neg^{\four} b)\,\};\\
(\eta_{\A})^{-1}(\{\,z\in\fnt{H}(\A_k)\mid z(a)=i\,\})&=\{\,x\in\D(\A)\mid x(a)\in\alpha^{-1}(1)\}\times\{\,\alpha\,\}\\
&\ \ \cup\{\,x\in\D(\A)\mid x(a)\in\beta^{-1}(1-i)\,\}\times\{\beta\,\}.
\end{align*}
Then $\eta_{\A}$ is a homeomorphism.
Hence, as claimed, 
 $\fnt{H}(\A_k)\cong (Y_{\A};\preccurlyeq',\Tp_{Y_{\A}})$. 
Since 
$(Y_{\A};\preccurlyeq',\Tp_{Y_{\A}})\cong \D(\A)\coprod_{\CP}\D(\A)$,
 we conclude that $\A_k\cong\Lalg\times\Lalg$, where~$\Lalg$ denotes the lattice~$\fnt{KD}(\A)$.

Theorem~\ref{Theo:RevEngDB} can be seen as 
the product representation theorem for distributive bilattices
expressed in dual form. 
We recall that, given a distributive lattice
 $\Lalg  =(L; \lor,\land,0,1)$,
then 
$\Lalg \odot \Lalg $ denotes the distributive bilattice with  universe 
$L \times L$ and lattice operations given by
\begin{alignat*}{2}
(a_1,a_2) \lor_t (b_1,b_2)  & = (a_1 \lor b_1, a_2 \land b_2), 
\quad & 
(a_1,a_2) \lor_k (b_1,b_2)  & = (a_1 \lor b_1, a_2 \lor b_2), \\
(a_1,a_2) \land_t (b_1,b_2)  & = (a_1 \land b_1, a_2 \lor b_2), \quad &
(a_1,a_2) \land _k (b_1,b_2)  & = (a_1 \land b_1, a_2 \land b_2);
\end{alignat*}
negation is given by $\neg(a) = (b,a)$
and the  constants by 
$0_t = (0,1)$, $1_t = (1,0)$,
$0_k = (0,0)$ and $1_k = (1,1)$.
A well-known example 
is the representation of $\four$ as $\two\odot\two$. More precisely, 
$h\colon\four\to \two\odot\two$ defined by $h(ij)=(i,1-j)$, for $i,j\in\{0,1\}$,  is an isomorphism.

As a consequence of Theorem~\ref{Theo:RevEngOps} we obtain the following result.

\begin{thm}\label{Theo:ProdFunDB}
Let $\fnt{V}\colon \DB\to \CCD$ and $\fnt{W}\colon \CCD\to \DB$ be the functors defined~by:
\begin{alignat*}{3}
&\text{on objects:} & \hspace*{2.25cm} &\A\longmapsto \fnt{V}(\A)=[0_k,1_t],
\hspace*{2.2cm}
\phantom{\text{on objects:}} &&\\
&\text{on morphisms:}  & & \, \,h \longmapsto \fnt{V}(h) =h{\restriction}_{[0_k,1_t]},
&& \\
\shortintertext{where $[0_k,1_t]$ is considered as a sublattice of $\A_{t}$ with bounds $0_k$ and $1_t$,  and} 
&\text{on objects:}  &  &\Lalg\longmapsto \fnt{W}(\Lalg) = \Lalg\odot\Lalg, &&\\
&\text{on morphisms:}  &  &\,\, g \longmapsto \fnt{W}(g)\colon (a,b)\mapsto (g(a),g(b)).  && 
\end{alignat*}

\noindent
Then 
$\fnt{V}$ and $\fnt{W}$ are naturally equivalent to  $\fnt{KD}$ and $\fnt{EH}$, respectively.
\end{thm}

\begin{coro}  {\rm (The Product Representation Theorem for 
distributive bilattices)} \label{DBprodrep}
Let $\A \in \DB$.   Then
there exists $\Lalg= (L; \lor,\land,0,1) \in \CCD$ such that
$\A \cong \Lalg \odot \Lalg$.
\end{coro}

We can now see the relationship between our natural duality for
$\DB$  and the dualities presented 
for this class 
in \cite{MPSV,JR}.  In \cite{MPSV}, the duality for $\DB$ is obtained by first proving that the product representation is part of an equivalence between the categories $\DB$ and $\CCD$. 
 The duality assigns 
to each 
$\A$ in $\DB$ the Priestley space $\fnt{H}([0_t,1_k])$, where the interval  $[0_t,1_k]$ is considered as a sublattice of $\A_t$. Then the functor
 from $\DB$ to $\CP$
 defined in \cite[Corollaries~12 and~14]{MPSV} corresponds to $\fnt{HV}$ where $\fnt{V}\colon\DB\to\CCD$ is as defined in Theorem~\ref{Theo:ProdFunDB}. 
The duality in \cite{JR},  
is arrived at by a different route.
At the object level, 
 the authors consider first the De Morgan reduct of a bilattice and then 
enrich its dual structure by adding   
two clopen 
up-sets  
of the dual which represent the constants $0_k$ and $1_k$. 
In the notation of Theorem~\ref{Theo:RevEngOps} their duality is based on the functor $\fnt{HU}^+$ by considering $\CA=\DB$ with only one lattice dual-endomorphism and two constants.  
The connection between their duality and ours follows from 
Theorems~\ref{Theo:RevEng} and~\ref{Theo:RevEngOps}. Firstly, Theorem~\ref{Theo:RevEng} tells
 us how to obtain $\fnt{L}$ from $\D$. 
Then Theorem~\ref{Theo:RevEngOps} shows how to enrich this functor to obtain $\fnt{L}^+$ and  confirms that the latter is naturally equivalent to
$\fnt{HU}^+$.


\begin{figure}  [ht]
\begin{center}
\begin{tikzpicture} [scale=.35]
 [inner sep=10mm,
 [auto,
 text depth=0.25ex,
 wiggly/.style={decorate,decoration={snake,amplitude=1pt,segment length=5pt,
                pre length=3pt,post length=3pt}}, 
]

\path (-3.5,2) 
node (topL) [scale=.4,shape=circle,draw]   {}
(-3.5,-3)
node (botL) [scale=.4,shape=circle,draw,gray,fill=gray]   {}
(2.5,-3)
node (botR) [scale=.4,shape=circle,draw,gray,fill=gray]  {}
(2.5,2)
node (topR) [scale=.4,shape=circle,draw]  {}
(-.5,8)
node (top1) [scale=.4,shape=circle,draw]  {}
(-.5,3)
node (bot1) [scale=.4,shape=circle,draw]  {}
(-.5,-4)
node (top0) [scale=.4,shape=circle,draw,gray,fill=gray]  {}
(-.5,-9)
node (bot0) [scale=.4,shape=circle,draw,gray,fill=gray]  {};

;
\path
 (-4.5,-2)
node   (swL)  [inner sep=0mm] {}
 (-2.5,-2)
node (seL)  [inner sep = 0mm] {}
 (-2.5,1) 
node   (neL) [inner sep=0mm] {}
(-4.5,1)
node (nwL)  [inner sep=0mm] {}
(1.5,-2)
node  (swR)   [inner sep=0mm] {}
 (3.5,-2)
node  (seR)   [inner sep=0mm] {}
 (3.5,1)
node  (neR)   [inner sep=0mm] {}
 (1.5,1)
node (nwR)   [inner sep=0mm]{}
(-1.5,4) 
 node (sw1)   [inner sep=0mm] {}
 (.5,4) 
node  (se1)   [inner sep=0mm] {}
(.5,7)
node  (ne1)   [inner sep=0mm] {}
 (-1.5,7) 
node (nw1)   [inner sep=0mm] {}
(-1.5,-8)
node (sw0)  [inner sep=0mm] {}
(.5,-8)
node (se0)   [inner sep=0mm] {}
 (.5,-5)
node (ne0)   [inner sep=0mm] {}
(-1.5,-5)
node (nw0)   [inner sep=0mm] {};

\draw   [thin] (topL) to (nwL);
\draw  [thin]  (topL) to (neL);  
\draw  [thin]  (botL) to (swL);
\draw   [thin] (botL) to (seL); 

\draw   [thin] (top1) to (nw1);
\draw  [thin]  (top1) to (ne1);  
\draw  [thin] (bot1) to (sw1);
\draw  [thin] (bot1) to (se1); 

\draw  [thin]  (top0) to (nw0);
\draw  [thin] (top0) to (ne0);  
\draw  [thin]  (bot0) to (sw0);
\draw  [thin]  (bot0) to (se0); 

\draw [thin] (topR) to (nwR);
\draw (topR) to (neR);  
\draw (botR) to (swR);
\draw (botR) to (seR);

\draw [thin] (-4.5,-2)--(-2.5,-2) --(-2.5,1)--(-4.5,1)--(-4.5,-2);
\draw  [thin] (1.5,-2)--(3.5,-2) --(3.5,1)--(1.5,1)--(1.5,-2);
\draw  [thin] (-1.5,4)--(.5,4) --(.5,7)--(-1.5,7)--(-1.5,4);
\draw [thin]  (-1.5,-8)--(.5,-8) --(.5,-5)--(-1.5,-5)--(-1.5,-8);

\draw [thin] (12,-2)--(14,-2) --(14,1)--(12,1)--(12,-2);
\draw  [thin] (14,-2)--(16,-2) --(16,1)--(14,1)--(14,-2);
\path
(12,-2)
node (swQ) [inner sep=0mm] {}
(12,1) 
node  (nwQ)  [inner sep=0mm] {}
(16,1) 
node (neQ)  [inner sep=0mm] {}
(16,-2) 
node  (seQ) [inner sep=0mm] {}
(14,4)
node (topQ) [scale=.4,shape=circle,draw,gray]  {}
(14,-4)
node (botQ) [scale=.4,shape=circle,draw,gray,fill=gray]  {};

\draw [thin] (seQ) to (botQ);
\draw [thin] (swQ) to (botQ);
\draw [thin] (neQ) to (topQ);
\draw [thin] (nwQ) to (topQ);

\draw [ultra thin,rounded corners] 
(-4.15,1.5) --(-2.2,8.5)--(.9,8.5)--(3.5,1.5)--(-4.3,1.5)--(-2.2,8.5);   
 
\draw [ultra thin,rounded corners] 
(-4.15,-2.5) --(-2.2,-9.5)--(.9,-9.5)--(3.5,-2.5)--(-4.3,-2.5)--(-2.2,-9.5);

\draw [very thin,->] (top0) to (botL);
\draw [very thin,->]  (top0) to (botR);
\draw [very thin,->]  (topL) to (bot1);
\draw [very thin,->]  (topR) to (bot1);

\draw [->,decorate,decoration={snake,amplitude=1pt,segment length=5pt,
                pre length=3pt,post length=3pt}]
(5,0)-- node [yshift=.5cm]  {$z \mapsto [z]_{\approx}$} (10.5,0);


\node (YA) at (4.5,-6) {$(Y_{\A};\preccurlyeq)$};
\node (HUA) at (14,-6) {$\Hu\Uu(\A)$};
\node (Yalpha) at  (-5.2,0) {$Y_{\alpha}$};
\node (Ybeta) at (.8,0) {$Y_{\beta}$};
\node (Yone) at (-2.1,6) {$Y_{\onebar}$};
\node (Yzero) at (-2.1, -6) {$Y_{\zerobar}$};
\node (leqk) at (-3.5,-1) {$\leq_k$};
\node (geqk) at (2.5,-1) {$\geq_k$};
\end{tikzpicture}
\end{center}
\caption{Obtaining $\Hu\Uu(\A)$ from $\Du(\A)$\label{fig:unbddquot}} 
\end{figure}

We now turn to  the unbounded case.
noting that, 
as regards dual representations, our results are entirely new,
since neither \cite{MPSV} nor \cite{JR} considers duality for unbounded distributive bilattices.  
We shall rely on 
Theorem~\ref{Theo:RevEngu}
to obtain a suitable  description of 
$\Ku\Du$ and $\Eu\Hu$. 
Fix $\A \in \DBU$ and let 
$Y_\w = \D(\A) \times \{ \w\}$, for 
$\w \in \Omega = \{\alpha,\beta,\zerobar,\onebar\}$.
Let $X$ be the doubly-pointed 
Priestley space obtained as in 
Theorem~\ref{Theo:RevEngu} by quotienting the pre-order 
$\preccurlyeq$ to obtain  a partial order.  
Note that $\D(\A)$ ordered by the pointwise lifting of $\leq_k$ has 
 top and bottom elements,
{\it viz.}~the constant maps onto~$10$ and  onto~$01$, respectively.   
Hence,  by Proposition~\ref{DBUpigsub}(i)(c)--(d), 
$Y_{\zerobar}$ 
collapses  to a single  point and is identified with the bottom point of
$Y_\alpha$ and the top point of $Y_\beta$.   In the same way, $Y_{\onebar}$
collapses to a point and is identified with the top point of $Y_\alpha$ and with the bottom point of $Y_\beta$. No additional identifications are made.  
This argument proves the following theorem.

\begin{thm} \label{Theo:RevEngDBU}  Let 
$\Du\colon\DBU\to\CP_{01}$ and $\Eu\colon\CP_{01}\to\DBU $ be the functors setting up the duality presented in Theorem~{\upshape\ref{DBUnatdual}}.
Then for each $\A\in \DBU$ the Priestley dual 
$\Hu(\A_t)$ of the $t$-lattice reduct of $\A$ is 
such that
\[
\textstyle\Hu(\A_t)\cong\Du(\A)\coprod_{\CP_{01}}\Du(\A)^\partial,
\]
where $\cong$ denotes an isomorphism of doubly-pointed Priestley spaces.
\end{thm}

Figure~\ref{fig:unbddquot}  illustrates the passage from 
$(\D(\A) \times \Omega;\preccurlyeq,\Tp)$ to $\Hu\Uu(\A)$,
including the way in which the union of the full set of piggybacking relations   supplies a pre-order.  
The pre-ordered 
set $(Y_{\A};\preccurlyeq)$ has as 
its universe four copies of 
$\D(\A)$.  Each copy is depicted in the figure by a linear sum of the form $\boldsymbol 1 \oplus P \oplus \boldsymbol 1$;  
 the top and bottom elements are depicted by circles.  For $Y_{\alpha}$,
$P$ carries the lifting of the partial order $r_{\alpha,\alpha}$, 
that is,
$\leq_k$ lifted 
to $\DBU(\A,\fourU)$; for $Y_{\beta}$ the corresponding order is 
the lifting of 
$\geq_k$
to $\DBU(\A,\fourU)$.  
Theorem~\ref{Theo:RevEngDBU}   shows that 
 $Y_{\onebar}$, together with the top elements of $(Y_\alpha; \leq_k)$ and of $(Y_\beta; \geq_k)$ form a single $\approx$-equivalence class,
and likewise all elements of $Y_{\zerobar}$ and the bottom elements of  
$Y_\alpha$ and of $Y_\beta$ form an $\approx$-equivalence class.  
These are the only $\approx$-equivalence class with more than one element.
Thus the quotienting map which yields $\Hu\Uu(\A)$ operates as shown.
 Topologically, the image $\Hu\Uu(\A)$ carries the quotient topology, 
 so that the top and bottom elements will both be isolated points if and only if 
$\A_t$ is a bounded lattice.

Theorem~\ref{Theo:RevEngDBU} states that $\Hu(\A_t)$ is obtained as the coproduct of the doubly-pointed Priestley spaces $\Du(\A)$ and $\Du(\A)^{\partial}$.   
This coproduct corresponds to 
the product of unbounded distributive lattices $\Lalg=\Ku\Du(\A)$ and $\Lalg^{\partial}$,  that is, $\A_t\cong \Lalg\times \Lalg^{\partial}$.
By  the same argument as 
in the bounded case, 
 $\A_k\cong \Lalg\times \Lalg$. Moreover, using 
the analogue of 
Theorem~\ref{Theo:RevEngOps},
we have 
\begin{alignat*}{2}
\bar{\neg}_{\A}([(x,\alpha)])&
=[(x,\beta)],
\quad \quad & 
\bar{\neg}_{\A}([(x,\beta)])&
=[(x,\alpha)];\\
\widehat{\neg^{\A}}(\alpha\circ x)&=\beta\circ x, &
\widehat{\neg^{\A}}(\beta\circ x)&=\alpha\circ x;\\
\widehat{\neg^{\A}}({\onebar}\circ x)&={\zerobar\circ x},
 &
\widehat{\neg^{\A}}(\zerobar\circ x)&=\onebar\circ x.
\end{alignat*}

The construction of $\Lalg\odot\Lalg$ 
for $\Lalg \in \CCD$ applies equally well to 
 $\Lalg\in\DU$; in this case the unbounded distributive bilattice $\Lalg\odot\Lalg$ is  defined on $L\times L$ by taking  $(\Lalg\odot\Lalg)_t=\Lalg\times\Lalg^{\partial}$,  $(\Lalg\odot\Lalg)_k=\Lalg\times\Lalg$ and  $\neg^{\Lalg\odot\Lalg}(a,b)=(b,a)$, for each $a,b\in L$.

Given  $\A\in\DBU$,  we define $\Lalg=\Ku\Du(\A)$. It follows 
from above 
that $\A\cong\Lalg\odot\Lalg$. Let  $h\colon\A\to \Lalg\odot\Lalg$ denote the isomorphism between $\A$ and $\Lalg\odot\Lalg$.  Then $\Lalg=\A_t/\ker(\rho)$ where $\rho(a)=a_1$ if $h(a)=(a_1,a_2)$. Using the~$\odot$ construction we  
observe that $(a,b)\in\ker(\rho)$ if and only if $a\land_t b=a\lor_k b$. This can also be proved using the fact that closed subspaces of 
doubly-pointed 
Priestley spaces correspond to congruences and that
\[
 \Hu(\Lalg)\cong
Y_{\alpha}=\Du(A)\times\{\alpha\}\cong 
Y_{\alpha}/_{\approx}
\subseteq
Y_{\A}/_{\approx}\cong\Du(\A)\textstyle\coprod_{\CP_{01}}\Du(\A)^\partial\cong \Hu(\A_t).
\]
 Now 
observe that  
the isomorphism $Y_{\A}/_{\approx}\cong \Hu(\A_t)$
 is determined by the unique $\CP_{01}$-morphism
 such that 
$(x,\w) \mapsto \w\circ x$, for $\w \in \{\alpha,\beta\}$,  and that  $\alpha$ is a $\DU$-homomorphism from $\A_t$ to $\twoU$  and also from  $\A_k$ to $\twoU$.  We deduce that   $(x\circ\alpha)(a)=(x\circ\alpha)(b)$ if and only if $a\land_t b=a\lor_k b$.

Our analysis yields the following theorem. 

\begin{thm}\label{Theo:ProdFunDBU}
For $\A \in \DBU$ let 
 $\theta_{\A}=\{\,(a,b)\in\A^2\mid a\land_t b=a\lor_k b\, \}$.  
Let $\unbounded{\fnt{V}}\colon \DBU\to \DU$ and $\unbounded{\fnt{W}}\colon \DU\to \DBU$ be the functors defined as follows:
  \begin{alignat*}{3}
&\text{on objects:} & \hspace*{1.9cm} & \A \longmapsto \unbounded{\fnt{V}}(\A) =\A_t/\theta_{\A}, \hspace*{1.9cm} \phantom{\text{on objects:}} &&\\
&\text{on morphisms:}  & &\,\, h\longmapsto \unbounded{\fnt{V}}(h)\colon 
[a]_{\theta_{\A}} \mapsto [h(a)]_{\theta_{\alg{B}}},
\text{ where 
$h\colon \A \to \B$,} &&\\
 \shortintertext{and}
& \text{on objects:}  & &\,\, \Lalg \longmapsto \unbounded{\fnt{W}}(\Lalg) =\Lalg\odot\Lalg,&& \\
&\text{on morphisms:}  & &\,  \, g \longmapsto \unbounded{\fnt{W}}(g) \colon (a,b)
\mapsto 
(g(a),g(b)). 
\end{alignat*}
Then 
$\unbounded{\fnt{V}}$ and $\unbounded{\fnt{W}}$ are naturally equivalent to  $\Ku\Du$ and $\Eu\Hu$, respectively.
\end{thm}

We have the following corollary;
cf.~\cite{RPhD,BJR11}.

\begin{coro} \label{Theo:prodrep-nobounds}  {\rm(Product 
Representation Theorem for unbounded distributive bilattices)}
Let $\A \in \DBU$.   Then there exists a distributive lattice $\Lalg$ such that $\A \cong \Lalg \odot \Lalg$.   Here the lattice $\Lalg$ may be 
 identified with the  quotient  $\A_i/\theta$, where~$\theta$ is the $\DU$-congruence given by 
$a\, \theta \, b$
if and only if $a\land_t b=a\lor_k b$.
\end{coro}

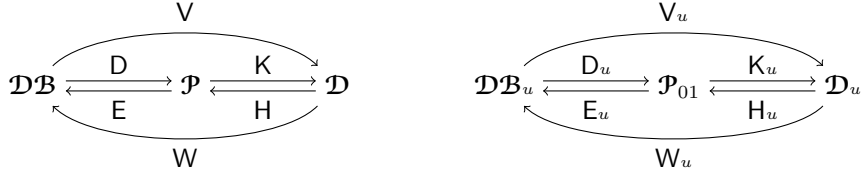
\begin{figure}  [h]
\begin{center}
\begin{tikzpicture}[scale=.85,
auto,
 text depth=0.25ex,
 move up/.style=   {transform canvas={yshift=1.9pt}},
 move down/.style= {transform canvas={yshift=-1.9pt}},
 move left/.style= {transform canvas={xshift=-2.5pt}},
 move right/.style={transform canvas={xshift=2.5pt}}] 
\matrix[row sep= 1cm, column sep= 1.42cm]
{ 
\node (DB) {$\DB$}; & \node (P) {$\CP$};  & \node (D) {$\CCD$}; &
\node (DBU) {$\DBU$}; & \node (P01) {$\CP_{01}$};  & \node (DU) {$\DU$};\\ 
};
\draw  [->, move up] (DB) to node  [yshift=-2pt] {$\D$}(P);
\draw [<-,move down] (DB) to node [swap]  {$\E$}(P);
\draw  [->,move up] (P) to node  [yshift=-2pt] {$\fnt{K}$}(D);
\draw [<-,move down] (P) to node [swap]  {$\fnt{H}$}(D);
\draw  [->, move up] (DBU) to node  [yshift=-2pt] {$\Du$}(P01);
\draw [<-,move down] (DBU) to node [swap]  {$\Eu$}(P01);
\draw  [->,move up] (P01) to node  [yshift=-2pt] {$\Ku$}(DU);
\draw [<-,move down] (P01) to node [swap]  {$\Hu$}(DU);
\draw [->] (DB) .. controls +(1,1) and +(-1,1) .. node {$\fnt{V}$} (D);
\draw [<-] (DB) .. controls +(1,-1) and +(-1,-1) .. node [swap] {$\fnt{W}$} (D);
\draw [->] (DBU) .. controls +(1,1) and +(-1,1) .. node {$\unbounded{\fnt{V}}$} (DU);
\draw [<-] (DBU) .. controls +(1,-1) and +(-1,-1) .. node [swap] {$\unbounded{\fnt{W}}$} (DU);
\end{tikzpicture}
\end{center}
\caption{The categorical equivalences in Theorems~\ref{Theo:ProdFunDB} and~\ref{Theo:ProdFunDBU}}\label{fig:Equiv}
\end{figure}

Figure~\ref{fig:Equiv}
summarises the categorical equivalences and dual equivalences 
involved in our approach, for both the bounded and unbounded cases.
As noted in the introduction,
our 
 approach leads directly to  categorical dualities,
 without the need to verify explicitly that the 
constructions are  functorial:  compare our presentation with that in  \cite[pp.~117--120]{MPSV}  and  note also the work carried out to set up  categorical equivalences on the algebra side in \cite[Section~5]{BJR11}.

\section{Applications of the natural dualities for $\DB$ and $\DBU$}\label{Sec:Applications}

In this section we demonstrate how the natural dualities we have 
developed 
so far  
lead easily to answers to questions of a categorical 
nature concerning $\DB$ and $\DBU$.
Using the categorical equivalence between $\DB$ and $\CCD$, and that between $\DBU$ and $\DU$, it is possible directly  to translate  certain concepts from one context to another. 
We shall concentrate on~$\DB$.  Analogous results can be obtained
for~$\DBU$ and we mention these explicitly only where this seems
warranted. 
We shall describe the following, in more or less detail: 
limits and colimits; 
 free algebras;  and  projective and injective objects.
These 
 topics are very traditional, 
and our aim is simply to show how our viewpoint allows
descriptions to be obtained, with the aid of duality, from
corresponding  descriptions in the context of  distributive lattices.  
The results we obtain here are new, but
unsurprising.
We shall also venture into 
territory less explored by duality methods and 
 consider    
 unification type,  and also admissible quasi-equations and clauses; here
substantially more work is involved.
It will be important for certain of the applications that we  
are dealing with  strong, rather than merely  full, dualities.
Specifically we shall make use of the fact that if functors $\D\colon \CA\to \CX$ and $\E\colon \CX \to \CA$ set up a strong duality then surjections (injections) in $\CA$
correspond to embeddings (surjections)  in $\CX$; see 
\cite[Lemma~3.2.6]{CD98}. 
On a technical point,  we note that we always assume that an algebra has  a non-empty universe.

\subsection*{Limits and colimits} \

Since $\DB$ is a variety, 
the forgetful functor into the category   $\cat{SET}$ of sets has a left adjoint. As a consequence all 
 limits in $\DB$ 
are calculated as in  $\cat{SET}$ (see \cite[Section~V.5]{McL}), and this renders them fairly easy to handle, with products being cartesian products and equalisers being calculated in $\cat{SET}$. 
(We refer the reader to \cite[Section~V.2]{McL} where the procedure to construct arbitrary limits from products and equalisers is fully explained.)

 The calculation of colimits is  more involved. 
The categorical equivalence between~$\DB$ and~$\CCD$ implies  that 
 if   
$\fnt{S}$  
is a diagram in $\DB$ then 
\[
\mathrm{Colim\,} 
\fnt{S}
\cong\E\fnt{H}\bigl(\mathrm{Colim\,}\fnt{KDS}
\bigr)\cong \fnt{W}\bigl(\mathrm{Colim\,}\fnt{VS}
\bigr).
\]
This observation transfers the 
problem from one category to the other, but does not by itself solve it.  
However we can  then 
use the natural duality derived  in
Theorem~\ref{DBnatdual}  
  in particular to compute  finite colimits. 
We  rely on the fact that 
colimits in $\DB$ 
correspond to limits in  $\CP$.  Such limits  
are easily calculated, since 
cartesian products and equalisers of Priestley spaces  are  again in $\CP$.
(Corresponding statements hold 
 for $\DBU$ and $\CP_{01}$~\cite[Section 1.4]{CD98}.)

Congruences can be seen as particular cases of colimits, specifically as co-equalisers.
 This implies, on the one hand, that the congruences of an algebra in $\DB$ or in $\DBU$ are in one-to-one correspondence with those substructures of its natural dual  that arise as equalisers. 
Since~$\DB$ is a variety and 
Theorem~\ref{DBnatdual} supplies a strong duality,
the lattice of congruences of an algebra~$\A$ in $\DB$ 
is dually isomorphic to the lattice of closed substructures of its dual space (see \cite[Theorem III.2.1]{CD98}). 
Simultaneously, the lattice of congruences of  ${\A\in\DB}$ 
is  isomorphic to the lattice of congruences of 
$\fnt{KD}(\A)\in\CCD$.
Likewise, from Theorem~\ref{DBUnatdual}, for each ${\A\in\DBU}$ the congruence  lattice of 
$\A$ is isomorphic to the congruence lattice of 
 $\Ku\Du(\A)\in\DU$. 
The latter  result was  
proved for interlaced bilattices in \cite[Chapter II]{RPhD} using the product representation.

 \subsection*{Free algebras}\

A natural duality
gives 
 direct access to a description of free objects:  
If an alter ego $\MT$ yields a duality on $\CA = \ISP(\M)$, then the 
power $\MT^\lambda $ is the natural
dual of the free algebra in $\CA$ on~$\lambda$ generators 
 (see \cite[Corollary II.2.4]{CD98}).
We immediately obtain
$     
\mathbf{F}_{\DB}(\lambda)\cong \E_{\DB}\bigl(\fourDBT^{\lambda}\bigr) 
$
where $\lambda$ is a cardinal and $\mathbf{F}_{\DB}(\lambda)$ 
denotes the free algebra on $\lambda$
 generators in $\DB$;  
the free generators correspond to the projection maps.

Because 
$\fourDBT=\twoT^2$, we have 
$\fnt{KD}(\mathbf{F}_{\DB}(\lambda))\cong\mathbf{F}_{\CCD}(2\lambda)$. 
Therefore 
\[\mathbf{F}_{\DB}(\lambda)\cong\fnt{EH}(\mathbf{F}_{\CCD}(2\lambda))\cong \mathbf{F}_{\CCD}(2\lambda)\odot\mathbf{F}_{\CCD}(2\lambda).\]
Hence 
$\mathbf{F}_{\CCD}(2\lambda)\odot\mathbf{F}_{\CCD}(2\lambda)$ is the free bounded distributive bilattice on~$\lambda$  generators,  the free generators being  the pairs $(x_{2i-1},x_{2i})$ where $\{\,x_i\mid i\in 2\lambda\,\}$ is the set of free generators of $\mathbf{F}_{\CCD}(2\lambda)$.
Analogous results hold for 
$\DBU$.

\subsection*{Injective and projective objects}\ 

Injective, projective and weakly projective objects in $\CCD$ have been described 
 (see 
\cite{BD} and the  references therein;  the definitions are given in Chapter I and results in  Sections~V.9 and  V.10). 
The notions of injective and projective object are preserved under categorical equivalences. 
For categories which 
are classes of algebras with homomorphisms as the morphisms,
weak 
projectives are also 
preserved under
categorical equivalences. 
A distributive lattice $\Lalg$ (with bounds)  is injective
in $\CCD$  (and in $\DBU$ too) if and only 
it is complete
and each element of $\Lalg$ is complemented  (see \cite[SectionV.9]{BD}). This implies  that  a  distributive bilattice 
$\A$ is injective 
in $\DB$ 
if and only if $\A_k$ 
is complete (or equivalently $\A_t$ is complete) and each element of $\A$ is 
complemented in 
$\A_k$ (or equivalently $\A_t$ is complete and each element of $\A$ is 
complemented in 
$\A_t$). 
Moreover, since~$\CCD$
has 
enough injectives,  
the same is true of  $\DB$.  Corresponding statements can be made for
$\DBU$. 

The algebra 
 $\two$ is the only projective of $\CCD$  \cite[Section V.10]{BD}.
Hence 
$\four$ is the only projective in $\DB$. The general description of 
weak
projectives in $\CCD$  
is rather involved (see \cite[Section V.10]{BD}). But in the case of finite algebras there is a simple dual characterisation: 
a 
finite
bounded distributive lattice is weakly projective in 
$\CCD$ if 
 and only if its dual space is a lattice. This 
translates to bilattices:
a finite 
 distributive bilattice is weakly projective in 
$\DB$
 if and only if its natural dual is a lattice, or equivalently if the family of
 homomorphisms into
$\four$, ordered pointwise by~$\leq_k$, forms  a lattice.
In the unbounded case we note that $\DBU$ has no projectives since
$\DU$ has none, and that  a finite member~$\A$ of $\DBU$ is weakly projective if and only if $\Du(\A)$ is a lattice.

\subsection*{Unification type} \

The notion of \defn{unification} was 
introduced by Robinson in \cite{Ro1965}. Loosely, (syntactic) unification 
 is the process of finding substitutions that equalise pairs of terms. When considering equivalence under an equational theory instead of equality the notion of unification evolves to 
encompass 
the concept of \defn{equational unification}. 
We refer the reader to  \cite{BS2001} for the general definitions and background theory of 
 unification. To study the unification type of bilattices we shall use 
 the notion 
of 
algebraic unification developed by Ghilardi in \cite{Gh97}. 

Let $\A$ be a finitely presented algebra in a quasivariety $\CA$. 
A \defn{unifier} for $\A$ in $\CA$ is a homomorphism $u\colon\A\to\alg{P}$, where $\alg{P}$ is a finitely generated weakly projective  algebra in $\CA$.
 (In \cite{Gh97} weakly projective algebras are called regular projective or simply projective.)  
An algebra $\A$ is said to be  \defn{solvable}  in $\CA$ if there exists at least one unifier for it. 
Let $u_i \colon \A \to \alg{P}_i$ for $i\in\{1,2\}$ be unifiers for $\A$ in $\CA$.
Then $u_1$ is \defn{more genera}l than~$u_2$, in symbols, $u_2 \leq u_1$,
if there exists a homomorphism $f \colon \mathbf{P}_1 \to \mathbf{P}_2$
such that $f \circ u_1=u_2$. 
A unifier $u$ for $\A$ is said  to be a \defn{most general} unifier (an mg-unifier) 
of   
$\A$ in $\CA$ if  $u \leq u'$ implies $u'\leq u$.
For $\A$ solvable in $\CA$ the \defn{type} of $\A$ is defined as follows:
\begin{longnewlist}
\item[\emph{nullary\!}] if there exists $u$, 
a  unifier of $\A$, such that $u\not\leq v$ for each mg-unifier of $\A$ (in symbols, $\mathrm{type}_{\CA}(\A)=0$);
\item[\emph{unitary\!}] if  there exists  a 
 unifier $u$ of $\A$ such that $v\leq u$ for each  unifier $v$ of~$\A$  ($\mathrm{type}_{\CA}(\A)=1$);
\item[\emph{finitary\!}] if 
there exists a finite set $U$ of  mg-unifiers of $\A$ such that for each unifier $v$  of $\A$ there exists $u\in U$ with $v\leq u$, and  for each $v$  of $\A$ there exists $w$ unifier of $\A$ with $w\not\leq v$ ($\mathrm{type}_{\CA}(\A)=\w$); and 
\item[\emph{\ infinitary\!\!}] otherwise ($\mathrm{type}_{\CA}(\A)=\infty$).
\end{longnewlist}
In \cite{BC}, an algorithm to classify finitely presented bounded distributive lattices by their unification type was presented. 
Since the unification type of an algebra is a categorical invariant (see \cite{Gh97}), the results in  \cite{BC} can be combined with the equivalence
 between $\DB$ and $\CCD$  to 
investigate  the unification types of 
finite distributive bilattices. 
 
Moreover, since the results in \cite{BC} were obtained using Priestley duality for~$\CCD$, we can directly translate the results to bilattices and their natural duals.
This yields 
the following characterisation.
Let $\A$ be a finitely presented (equivalently, finite) bounded distributive bilattice. Then $\A$ is solvable in $\DB$ if and only if it is non-trivial and 
\[
\mathrm{type}_{\DB}(\A)=\begin{cases}
1&\mbox{ if }\D_{\DB}(\A)\mbox{ is a lattice, i.e., 
 if }\A\mbox{ is weakly projective,}\\
\w&\mbox{ if }\D_{\DB}(\A)\mbox{ is not a lattice and}\\
&\mbox{ \hspace*{.3cm} for each } x,y\in\D_{\DB}(\A)\mbox{ the interval } [x,y]\mbox{ is a lattice,}\\
0&\mbox{ otherwise.}
\end{cases}
\]

In \cite{BC} the 
corresponding theory for
 unbounded distributive lattices was not developed. 
With minor modifications to the proofs presented there, it is easy to  extend the results
 to $\DU$. Its translation to $\DBU$ is as follows.
Each finite algebra $\A$ in $\DBU$ is solvable and
\[
\mathrm{type}_{\DBU}(\A
)=\begin{cases}
1&\mbox{ if }\D_{\DBU}(\A)\mbox{ is a lattice, i.e., 
 if }\A\mbox{ is weakly projective,}\\
0&\mbox{ otherwise.}
\end{cases}
\]

\subsection*{Admissibility}\

The concept of admissibility was  
introduced by Lorenzen for intuitionistic logic \cite{Lo1955}. Informally, 
a rule is admissible in a logic if 
when 
the rule is 
added to the system it does not modify the notion of theoremhood. The study of admissible rules for logics that admit an algebraic semantic
has led to the investigation 
of admissible rules for equational logics of classes of algebras. For  background 
on admissibility  we refer the reader to \cite{Ry97}. 

A \defn{clause} in an algebraic language $\mathcal L$ is an ordered pair of finite sets of $\mathcal L$-identities, written $(\Sigma, \Delta)$. 
Such a  
clause is called 
a \defn{quasi-identity} if $\Delta$ contains only one identity. 
Let $\CA$ be a quasivariety of algebras with language  $\mathcal L$. We say that the $\mathcal L$-clause $(\Sigma,\Delta)$ is \defn{valid }in $\CA$ (in symbols $\Sigma\vDash_{\CA}\Delta$) if for 
every $\A \in \CA$ and homomorphism $h \colon {\alg{Term_{\mathcal L}}} \to \A$,
we have that 
$\Sigma \subseteq \ker h$ implies 
$\Delta \cap \ker h \not = \emptyset$, where ${\alg{Term_{\mathcal L}}}$ denotes the term (or absolutely free) algebra for~$\mathcal L$ over 
countably 
many variables (we are assuming that $\Sigma\cup\Delta\subseteq {\alg{Term_{\mathcal L}}}^2$). 
For simplicity we shall work with  the following equivalent definition
of admissible clause:
the clause $(\Sigma,\Delta)$ is called 
\emph{admissible in~$\CA$} if it is valid in the free $\CA$-algebra on countably many  generators,  ${\alg{F}_{\CA}}(\aleph_0)$.

Let $\CA$ be a quasivariety. If a set of quasi-identities $\Lambda$ is such that $\A\in\CA$ 
belongs to 
the quasivariety generated by ${\alg{F}_{\CA}}(\aleph_0)$
if and only if $\A$ satisfies the  quasi-identities in  $\Lambda$,  then   $\Lambda$ is called a \defn{basis for 
the admissible quasi-identities of $\CA$}. 
Similarly, 
$\Lambda$ is called a \defn{basis for the admissible clauses of  $\CA$} if $\A$ satisfies the  clauses in  $\Lambda$ if and only if $\A$ is in the universal class generated $\alg{F}_{\CA}(\aleph_0)$, that is, $\A$ satisfies the same clauses 
as $\alg{F}_{\CA}(\aleph_0)$ does.

In the case of a locally finite quasivariety,  checking that a set of clauses or quasi-identities is a basis can be restricted to finite algebras.

\begin{lem}
\normalfont{\cite{CM}}  \label{Lem:CM} 
Let $\CA$ be a locally finite quasivariety and let $\Lambda$ be  a set of clauses in the language of $\CA$. 
\begin{newlist}
\item[{\rm (i)}] 
The following statements are equivalent:
\begin{newlist}
\item[{\rm (a)}]	 for each finite $\A \in \CA$
it is the case that 
 $\A \in \ope{IS}(\alg{F}_{\CA}(\aleph_0))$ if and only if  
$\A$ satisfies~$\Lambda$;
\item[{\rm (b)}] 	$\Lambda$ is a basis for the admissible clauses of $\CA$.
\end{newlist}		
\item[{\rm (ii)}] 
If the set $\Lambda$ consists of quasi-identities, then the following 
statements
are equivalent:
\begin{newlist}
\item[{\rm (a)}]	 for each finite $\A \in \CA$ it is the case that $\A \in \ISP(\alg{F}_\CA(\aleph_0))$ if and only if  $\A$ satisfies~$\Lambda$;
\item[{\rm (b)}] 	$\Lambda$ is a basis for the admissible quasi-identities of~$\CA$.
\end{newlist}
\end{newlist}	
\end{lem}

In \cite{CM},
using this lemma and the appropriate natural dualities,   bases for admissible quasi-identities and clauses were presented for  various classes of algebras---bounded distributive lattices, Stone algebras and De Morgan algebras, among others. Here we follow the same strategy using the dualities for $\DB$ and $\DBU$ developed in Sections~\ref{sec:DB} and~\ref{sec:DBU}.
\begin{lem}\label{Lem:AdmDB}
Let $\A$ be a finite 
distributive bilattice. 
\begin{newlist}
\item[{\rm (i)}] $\A\in \ISP(\alg{F}_{\DB}(\aleph_0))$.
\item[{\rm (ii)}] The following statements are equivalent:
\begin{newlist}
\item[{\rm (a)}] $\A\in \ope{IS}(\alg{F}_{\DB}(\aleph_0))$;
\item[{\rm (b)}] $\D_{\DB}(\A)$ is a non-empty bounded poset;
\item[{\rm (c)}] $\A$ satisfies the following clauses:
  \begin{newlist}
\item[{\rm (1)}] $  ( \{ x\land_k y\approx 1_t \}, \{x\approx 1_t,\ y\approx1_t\} )$,
\item[{\rm (2)}] $ ( \{ x\lor_k y\approx 1_t\}, \{x\approx 1_t, \ y\approx 1_t\} )$,
\item[{\rm (3)}] $ ( \{0_t=1_t\}, \emptyset )$.
  \end{newlist}
\end{newlist}
\end{newlist}
\end{lem}
\begin{proof}
To prove (i) it is enough to observe that $\four$ is a subalgebra of any non-trivial algebra in~$\DB$, and therefore  $\DB=\ISP(\four)\subseteq\ISP(\alg{F}_{\DB}(\aleph_0))\subseteq\DB$.

To prove (ii)(a)$\Rightarrow$(ii)(b), let $h\colon\A\to \alg{F}_{\DB}(\aleph_0)$ be an injective homomorphism. Then the map
 $
\D_{\DB}(h)\colon \D_{\DB}(\alg{F}_{\DB}(\aleph_0))\to \D_{\DB}(\A)
$
 is an order-preserving continuous and onto $\D_{\DB}(\A)$. Since  $ \D_{\DB}(\alg{F}_{\DB}(\aleph_0))\cong \fourDBT^{\aleph_0}$ is bounded and non-empty,  
so is $\D_{\DB}(\A)$.

We next  prove the converse, namely   (ii)(b) $\Rightarrow$ (ii)(a). Let $\mathbf{t},\mathbf{b}\colon\A\to\four$ be the top and bottom elements of $\D_{\DB}(\A)$ and let $\{\mathbf{t},\mathbf{b},x_1,\ldots, x_n\}$ be an enumeration of the elements of the finite set $\D_{\DB}(\A)$. Let $\spc{P}=\fourDBT^{n}$, then $\E_{\DB}(\spc{P})$ is the free bounded distributive bilattice on $n$ 
generators. Then $\E_{\DB}(\spc{P})$ belongs to $\ope{IS}(\alg{F}_{\DB}(\aleph_0))$. Now  define 
$f\colon \spc{P}\to \D_{\DB}(\A)$ 
by 
\[
f(c_1,\ldots,c_{n})=\begin{cases}
	\mathbf{b} & \mbox{ if }c_i=0_k\mbox{ for each }i\in\{1,\ldots,n\},\\
	x_i & \mbox{ if }c_i\neq0_k,\mbox{ and } c_j=0_k\mbox{ for each }j\in\{1,\ldots,n\}\setminus\{i\}, \\
	\mathbf{t} & \mbox{ otherwise.}\\
\end{cases}
\]
It is easy to check that $f$ is order-preserving and maps  $P$ onto 
$\D_{\DB}(\A)$. Since the natural duality of 
Theorem~\ref{DBnatdual} is strong,  the dual homomorphism $\E_{\DB}(f)\colon\fnt{ED}(\A)\to \E_{\DB}(\spc{P})$ is injective. 
Hence  
\[\A\cong\fnt{ED}(\A)\in\ope{IS}(\E_{\DB}(\spc{P}))\subseteq\ope{IS}(\alg{F}_{\DB}(\aleph_0)).\]
 
We now prove  (ii)(b) $\Rightarrow$ (ii)(c). Let $\mathbf{t}\colon\A\to\four$ be the top element of $\D_{\DB}(\A)$ and assume that $a,b\in\A$ are such that $a\land_k b=1_t$. If we assume that $a\neq 1_t \neq b$ then  there exist $h_1,h_2\colon \A\to \four$ such that $1_t<_k h_1(a)$ and $1_t<_k h_2(b)$. Since the order in $\D_{\DB}(\A)$ is determined pointwise by $\leq_k$, we then have  $1_t <_k  \mathbf{t}(a),\mathbf{t}(b)$. Then $\mathbf{t}(a)=\mathbf{t}(b)=1_k$ and $\mathbf{t}(a\land_k b)=1_k\neq 1_t$, a contradiction. Then $a= 1_t$ or $b=1_t$.  A similar argument proves that  $\D_{\DB}(\A)$ having a lower bound implies that clause (2) is valid in $\A$.
If $\A\in\DB$ is such that $0_t=1_t$ then $\A$ is trivial and $\D_{\DB}(\A)$ is empty. This proves that clause~(3) is valid in any  algebra $\A$ whose natural dual $\D_{\DB}(\A)$ is non-empty.
  
Finally we prove  (ii)(c) $\Rightarrow$ (ii)(b). Let $F=\{\,c\in\A\mid 1_t\leq_k
 c\,\}$.
 By clause~(3), $\A$ is non-trivial,  so  
$0_t\notin F$. By clause (2), $F$ is a prime $k$-filter and it contains~$1_t$. 
Thus  
it is 
a prime $t$-filter,
as observed at the end of Section~\ref{sec:DBilat}.

Let $x\colon\A\to\two$ be the characteristic 
function 
of~$F$. Then the map $f\colon\A\to \four$ defined for each $a\in A$ by $f(a)=x(a)(1-x(\neg a))$ 
is a well-defined bilattice homomorphism, as observed after Theorem~\ref{sep-prop-bdd}. We shall prove that $f$ is the bottom element of $\D_{\DB}(\A)$. Let $h\in\D_{\DB}(\A)$ and $a\in A$. If
$a\in F$ and $\neg a \notin F$, since $1_t\leq_k a$, then $f(a)=1_t\leq_k h(a)$.  
If $a,\neg a\in F$, then $1_t\leq_k h(a),\neg h(a)$. Then $h(a)=1_k=f(a)$. 
The other two cases follow by a similar argument, since $1_t\leq_k a,\neg a$.  Then $f(a)\leq_k h(a)$ for each $a\in A$. This proves that $f\leq h$ in $\D_{\DB}(\A)$.

By a similar argument  
the validity of clause (1) implies that $\D_{\DB}(\A)$ is upper-bounded. 
\end{proof}
Combining Lemmas~\ref{Lem:CM} and~\ref{Lem:AdmDB} we obtain the following theorem.
\begin{thm}
Every admissible quasi-equation in $\DB$ is also valid in $\DB$. 
Moreover the following clauses form a basis for the admissible clauses for $\DB$
\begin{multline*} 
 ( \{ x\land_k y\approx 1_t \}, \{x\approx 1_t,\ y\approx1_t\} ), \ \ 
( \{ x\lor_k y\approx 1_t\}, \{x\approx 1_t, \ y\approx 1_t\} ) \\
\mbox{ and }\
 (\{0_t=1_t\},\emptyset).
\end{multline*} 
\end{thm}

To simplify the proof of Lemma~\ref{Lem:AdmDB}
 the clauses presented in the previous theorem used
 the  $k$-lattice operation.
 We can use Lemma~\ref{90deg} to rewrite the clauses using only constants and $t$-lattice operations.
\begin{lem}\label{Lem:AdmDBU}
Every finite unbounded
distributive bilattice $\A$ is isomorphic to a subalgebra of $\alg{F}_{\DBU}(\aleph_0)$.
\end{lem}
\begin{proof}
Let $\Du(\A)=(X;\leq,\top,\bot,\Tp)$. Since we assume that every algebra is non-empty,
 $X$ is non-empty. Let $X=\{\top,\bot,x_1,\ldots,x_n\}$  be an enumeration of the elements of $X$. Let $\spc{Q}=(\fourDBTU)^{n}$.
 Then $\Eu(\spc{Q})$ is the free distributive bilattice on $n$ generators and it belongs to $\ope{IS}(\alg{F}_{\DBU}(\aleph_0))$. 
Define $f\colon \spc{Q}\to \D_{\DB}(\A)$
by 
\[
f(c_1,\ldots,c_{n})=\begin{cases}
	\bot & \mbox{ if }c_i=0_k\mbox{ for each }i\in\{1,\ldots,n\},\\
	x_i & \mbox{ if }c_i\neq0_k\mbox{ and } c_j=0_k\mbox{ for each }j\in\{1,\ldots,n\}\setminus\{i\} ,\\
	\top & \mbox{ otherwise.}\\
\end{cases}
\]
Then $f$ is a continuous order-preserving map
with image  
$\Du(\A)$.  
Since the duality presented in Theorem~\ref{DBUnatdual} is strong,
 $\Eu(f)\colon\Eu\Du(\A)\to \Eu(\spc{Q})$ is injective. Then  $\A\in\ope{IS}(\alg{F}_{\DBU}(\aleph_0))$.
\end{proof}
The  
following theorem follows directly from  Lemmas~\ref{Lem:CM} and~\ref{Lem:AdmDBU}.
\begin{thm}
Every admissible clause in $\DBU$ is also valid in $\DBU$.
\end{thm}

\section{Multisorted natural dualities} \label{sec:multi}
 We have delayed presenting dualities for  pre-bilattice varieties  because, 
to fit  $\DPB$ and $\DPBU$ 
into our general representation scheme,
we shall draw on  the multisorted  version of 
natural duality theory. This  
originated in \cite{DP87} and is summarised in
\cite[Chapter 7]{CD98}. 
It  is applicable in particular to the situation that interests us, 
in which we 
have a quasivariety 
 $\CA = \ISP(\M_1,\M_2)$, where~$\M_1$ and $\M_2$  are non-isomorphic finite algebras of common type 
having a reduct in $\CCD$ or $\DU$.
We 
require 
the 
 theory only for algebras
$\M_1$ and~$\M_2$ of size~two.
We do not
set up the machinery of piggybacking, opting  
instead to work  with the multisorted version of the NU Duality Theorem, as given in \cite[Theorem~7.1.2]{CD98},   in a  form 
adequate 
to yield  strong dualities
for $\DPB$ and $\DPBU$.  We now give 
just enough 
information to  enable us to formulate
 the results we require.
 The ideas parallel
those presented in
Section~\ref{piggyonesorted}. 

Given 
 $\CA = \ISP(\M_1,\M_2)=\ISP(\CM)$,
we shall  initially consider an alter ego  for $\CM$ which takes  the form 
 $
   \CMT = (M_1 \du M_2; R, \Tp)$,
where $R$ is  a  set of relations each of which is a subalgebra 
of some $\M_i \times \M_j$, where $i,j \in \{1,2\}$.  (To obtain a strong duality
 we may need to
allow for 
nulllary  
operations 
as well, but for simplicity we defer  
introducing this refinement.)
The alter ego $\CMT$ is given the disjoint union topology derived from 
the discrete topology on $\M_1$ and $\M_2$.  
 We may then form multisorted topological structures~$\X = \X_1 \du \X_2 $ where each of the sorts $\X_i$ is a Boolean topological space,~$\X$ is equipped with the disjoint union topology
and, regarded as a structure, $\X$ carries a set 
$R^\X$ of relations $r^\X$; if $r \subseteq M_i \times M_j$,
then $r^\X \subseteq X_i \times  X_j$.
 Given 
structures $\X $ and $\Y$ in $\CX$, a morphism $\phi \colon \X \to \Y$
is a continuous map preserving the sorts, so that 
$\phi(X_i) \subseteq Y_i$, and  
 $\phi$ preserves 
the relational structure.
The terms isomorphism, 
embedding, etc., extend in the obvious way to the multisorted setting.

We define our dual category $\CX$ to have 
as objects those structures $\X$ 
which belong to 
$\IScP(\CMT)$.  Thus $\CX$ consists of isomorphic copies of closed substructures of powers of $\CMT$; here powers are formed  `by sorts'; 
and   the 
relational structure is lifted pointwise to substructures of such powers in the expected way.
We now define the hom-functors that will set up our duality.  Given 
$\A \in \CA$ and  we let
 $ 
 \D(\A) = \CA(\A,\M_1) \du \CA(\A,\M_2)$,
where  $\CA(\A,\M_1) \du \CA(\A,\M_2)$ is a (necessarily closed) substructure of $M_1^A
\du
M_2^A$ with the relational structure defined pointwise.
Given $\X = \X_1 \du \X_2 \in \CX$,  
we may form the set $\CX(\X,\CMT)$ of 
 $\CX$-morphisms from $\X$ into $\CMT$. 
This set 
acquires the structure of a member of 
$\CA$ by virtue of viewing it as a subalgebra of the power $\M_1^{X_1} 
 \times   
\M_2^{X_2}$. 
We 
define
$\E(\X) = \CX(\X,\CMT)$.
Let  $\D$ and $\E$  act on morphisms by composition in the obvious way.
We then have well-defined functors $\D \colon \CA \to \CX$ 
and $\E \colon \CX \to \CA$.  
We say $\CMT$ \defn{yields a multisorted duality} if, for each 
$\A \in \CA$,  the natural 
multisorted evaluation map $e_{\A}$ given by $e_{\A} (a) \colon x \mapsto x(a)$
is an isomorphism from~$\A$ to $\ED(\A)$.   The duality is \defn{full} if 
each evaluation map $\varepsilon_{\X} \colon \X \to \D\E(\X)$
is an isomorphism.  As before 
 we do not 
present the definition of strong duality,
noting only that a strong duality is 
necessarily full.
The following very restricted form of \cite[Theorem~7.1.1]{CD98}  will 
meet  our needs.

\begin{thm}  {\rm(Multisorted NU Strong Duality Theorem, special case)}\label{Theo:NuTwoSorts}
Let $\CA = \ISP(\M_1,\M_2)$, where $\M_1,\M_2$ are  
two-element 
 algebras of common type having lattice reducts.  Let 
$
\CMT = (\,M_1 
\du 
 M_2 ;  R, N,\Tp\, ) 
$
where 
 $ N $ contains all one-element subalgebras of 
$\M_i$, for $i=1,2$, treated as nullary operations, 
$ 
R$ is the set $\bigcup 
\{\, \Su(\M_i\times\M_j)\mid i,j\in\{1,2\}\,\}$,
and $\Tp$ is
is the disjoint union topology obtained from  the discrete topology 
on
$M_1$ and $M_2$.  
Then $\CMT$ yields a multisorted
duality on $\CA$ which is strong.
\end{thm}

\section{Dualities for distributive pre-bilattices} \label{sec:DPB}

Paralleling our treatment of other varieties, we first record the 
result on the structure of $\DPBU$ and $\DPB$ we shall require.

\begin{prop} \label{sep-pre-bilat}  
{\rm (i)}  $\DPBU = \ISP(\twoU^+,\twoU^-)$ and 
{\rm (ii)} 
 $\DPB = \ISP(\two^+,\two^-)$.
\end{prop}

\begin{proof}
Let $\A \in \DPBU$ and let $a \ne b$ in $\A$.  Since $\DU= \ISP(\twoU)$, there exists $x \in \DU(\A_t,\twoU)$ with 
$x(a) \ne x(b)$.  The relation $\theta$ given by  $c \, \theta \,  d$ if and only if $x(c) = x(d)$ 
is a $t$-lattice congruence and hence, by
Proposition~\ref{lem:cong2},
a $\DPBU$-congruence.   The associated quotient algebra has two elements, and is necessarily (isomorphic to) either $\twoU^+$ or 
$\twoU^-$.  This proves~(i).
The same form of 
argument works for (ii),  the only difference being that the map~$x$ 
now also preserves 
bounds.    
\end{proof}

The following two theorems are consequences of the Multisorted 
NU Duality Theorem.  We consider $\DPB$ first since the absence of 
one-element subalgebras makes matters particularly simple. 
We tag elements  with $\pm$ to indicate which $2$-element algebra they belong to.  In both cases we could use
either $\leq_k$ or $\leq_t$ as the subalgebra of the square in 
either component.  The choice we make mirrors that  forced 
when 
negation is present.  The choice  will affect how the translation to the Priestley-style duality operates, but not
the resulting duality.

\begin{thm}  \label{Theo:DPBduality}
A 
strong 
natural
 duality for  $\DPB = \ISP(\two^+, \two^-)$ is obtained
as follows.  Take $\CM = \{ \two^+,\two^-\}$ and 
as  the alter ego 
\[
\CMT = (\{ 0^+,1^+\} 
\du 
\{ 0^-,1^-\}; r^+, r^-, \Tp),
\]
where
 $r^+$  is $\leq_k$ on $\two^+$ 
and 
$r^-$ is  
$\leq_k $ on $\two^-$.
 Moreover 
$\DPB$ is dually equivalent to the category $\CX=\IScP(\CMT )$.
\end{thm}

\begin{proof} 
The algebras $\two^+$, $\two^-$, $\two^+\times\two^-$ and $\two^-\times\two^+$ have no proper 
subalgebras.  The proper subalgebras of $\two^+\times\two^+$ are
the diagonal subalgebra $\{(0,0),(1,1)\}$, and  $\leq_k$ and 
its converse,  
and likewise for  $\two^-\times\two^-$.
\end{proof}

Let $\CMT $ and $\CX$ be 
as in Theorem~\ref{Theo:DPBduality}.  
Since $r^+$ and $r^-$ are partial orders on the respective sorts, $(X_1,X_2;\leq_1,\leq_2,\Tp)$ belongs to 
$\IScP(\CMT)$ 
 if and only if the topological posets
 $(X_1,\leq_1,\Tp{\upharpoonright}_{X_1})$ and $(X_2,\leq_2,\Tp{\upharpoonright}_{X_2})$ are Priestley spaces. Moreover, since the morphisms in $\CX$ are continuous maps that preserve the sorts and both relations, 
we deduce
 that 
 a categorical equivalence between~$\CX$  and $\CP\times\CP$ is set up
by 
the functors $\fnt{F}\colon \CX\to \CP\times\CP$ and $\fnt{G}\colon \CP\times\CP\to \CX$  defined~by
\begin{alignat*}{3}
&\text{on objects:} & \hspace*{.95cm} &
\X= (X_1\du X_2;\leq_1,\leq_2,\Tp)\, \longmapsto \fnt{F}(\X) =
\hspace*{.95cm}\phantom{\text{on objects:}}&&
\\
       &&      &  \hspace*{3.5cm} \bigl( (X_1;\leq_1,\Tp{\upharpoonright}_{X_1}), (X_2;\leq_2,\Tp{\upharpoonright}_{X_2}) \bigr),&&
\\
&\text{on morphisms:}  & &
\phantom{\X =(X_1
\du X_2;\leq_1,\leq_2,) } 
\, h \longmapsto \fnt{F}(h) = 
(h{\upharpoonright}_{X_1},h{\upharpoonright}_{X_2}), 
&&\\
\shortintertext{and} 
&\text{on objects:} & &
\Z = 
 \bigl((X;\leq_X,\Tp_X),(Y;\leq_Y,\Tp_Y) \bigr)\,\longmapsto  \fnt{G}(\Z) = && \\
 & &  & \hspace*{5.5cm} (X\du Y;\leq_X,\leq_Y, \Tp),  && \\
&\text{on morphisms:} & &
  \phantom{
 X;\leq_X,\Tp_X),(Y;\leq_Y,\Tp_Y)\,\,\, } 
 (f_1,f_2)\longmapsto  \fnt{G}(f_1,f_2) = f_1\du f_2, &&
\end{alignat*}
where $\Tp$ is the topology on $X\du Y$ generated by $\Tp_X\du\Tp_Y$.
Then the diagram in Figure~\ref{fig:EquivDPBCCDCCD} proves that $\DPB$ is categorically equivalent to $\CCD\times\CCD$, where $\fnt{H}\times\fnt{H}$ and $\fnt{K}\times\fnt{K}$ are the corresponding product functors.
\begin{figure}  [ht]
\begin{center}
\begin{tikzpicture} 
[auto,
 text depth=0.25ex,
 move up/.style=   {transform canvas={yshift=1.9pt}},
 move down/.style= {transform canvas={yshift=-1.9pt}},
 move left/.style= {transform canvas={xshift=-2.5pt}},
 move right/.style={transform canvas={xshift=2.5pt}}] 
\matrix[row sep= 1cm, column sep= 1.7cm]
{ 
\node (DPB) {$\DPB$}; & \node (X) {$\CX$}; & \node (PP) {$\CP\times\CP$};  & \node (DD) {$\CCD\times\CCD$};\\ 
};
\draw  [->, move up] (DPB) to node  [yshift=-2pt] {$\D$}(X);
\draw [<-,move down] (DPB) to node [swap]  {$\E$}(X);
\draw  [->,move up] (X) to node  [yshift=-2pt] {$\fnt{F}$}(PP);
\draw [<-,move down] (X) to node [swap]  {$\fnt{G}$}(PP);
\draw  [->,move up] (PP) to node  [yshift=-2pt] {$\fnt{K}\times\fnt{K}$}(DD);
\draw [<-,move down] (PP) to node [swap]  {$\fnt{H}\times\fnt{H}$}(DD);
\end{tikzpicture}
\end{center}
\caption{Equivalence between $\DPB$ and
$\CCD \times \CCD$
 }\label{fig:EquivDPBCCDCCD}
\end{figure}

 To obtain   a strong  
duality for $\DPBU$ we need first to determine $\Su(\M)$ and $\Su(\M\times\M')$ where $\M,\M'\in\{\twoU^+, \twoU^-\}$.   
To determine 
which binary relations to 
include we can argue 
in much the same way 
as for $\Su(\fourU^2)$.   
Decomposable subalgebras of 
$\Su(\M\times\M')$ can be discounted.  It is
simple  
to confirm that  
all  indecomposable $\DPBU$-subalgebras
are 
$\DPB$-subalgebras, and such subalgebras 
have already been 
identified in the proof of Theorem~\ref{Theo:DPBduality}.   
We omit the details.

\begin{thm}  \label{Theo:DPBUduality}
A strong, and hence full,
duality for  $\DPBU = \ISP(\twoU^+, \twoU^-)$ is obtained
as follows.  Take $\CM = \{ \twoU^+,\twoU^-\}$ and 
as  the alter ego 
\[
\CMT = (\{ 0^+,1^+\} \cup \{ 0^-,1^-\}; r^+, r^-, 0^+,  1^+, 
0^-,  1^-,\Tp),
\]
where $r^+$  is 
$ \leq_k$ on $\twoU^{+}$ and 
$r^-$ is   $\leq_k$ on $\twoU^{-}$ 
and the  constants are treated as 
 nullary operations.
\end{thm}

Reasoning as in the bounded case, $\CX=\IScP(\CMT)$ is categorically equivalent to $\CP_{01}\times\CP_{01}$. 
Then $\DPBU$ is categorically equivalent to $\DU\times\DU$.  We have an exactly parallel situation to that shown in the diagram in Figure~\ref{fig:EquivDPBCCDCCD}.

As an aside, 
we 
 remark 
that we could generate $\DPBU$ as a quasivariety
using the single generator ${\twoU^+ \times \twoU^-}$ and apply
Theorem~\ref{genpigoneMu}. 
But there are some merits in working with the pair of 
algebras $\twoU^+$ and $\twoU^-$.  Less work is involved to formulate a strong duality and to confirm that it is indeed strong.  More importantly for our purposes,
the translation to a Priestley-style duality is more transparent in the multisorted
framework.

As  was done in Theorems~\ref{Theo:ProdFunDB} and~\ref{Theo:ProdFunDBU} for $\DB$ and $\DBU$, respectively, it is possible to develop a different presentation (naturally equivalent) of the functors that determine the 
equivalences between 
 $\CCD\times\CCD$ and $\DPB$ and between $\DU\times\DU$ and $\DPBU$. This will lead to the known product decomposition of 
distributive pre-bilattices with and without bounds. We choose not to develop this here, since 
we would need to introduce the multisorted version of the piggyback duality (see \cite[Theorem~7.2.1]{CD98}).
The results could then  be  obtained just by modifying the arguments  used to prove Theorems~\ref{Theo:ProdFunDB} and~\ref{Theo:ProdFunDBU}. 
Also the applications presented in Section~\ref{Sec:Applications} can be extended to $\DPB$ and $\DPBU$ with the corresponding modifications.

\section{Concluding remarks}  \label{Sec:Conclude}

With our treatment of representation theory for distributive bilattices now complete, we can take stock of what we have achieved.

The scope of our work is somewhat different 
from that of other investigators of  bilattices.
Throughout  
 we have restricted attention 
to the distributive case.  
We have  not ventured into the territory of logical bilattices in this paper, but  we do observe that such bilattices are customarily assumed to be distributive.   
Nevertheless we should 
comment on the role of distributivity, as compared with the
weaker condition of interlacing. 
 Any interlaced (pre)-lattice has a product representation and, conversely, 
  such a representation is available only if the two lattice structures
are linked by interlacing.  Accordingly the product representation 
features very strongly in the literature.
As 
 indicated in Section~\ref{sec:prodrep}, the dual representations 
obtained in \cite{MPSV} and in \cite{JR}  build on 
Priestley duality as it applies to the varieties 
$\CCD$ and $\cat{DM}$.  The setting, perforce, is now that in which 
the bilattice structures are distributive and have bounds;  the product representation is brought into play to  handle the $k$-lattice operations.    

We  next comment 
 on the role of congruences.
In this paper,
the core result is Proposition~\ref{lem:cong2} asserting  
that the  congruences of any distributive  pre-bilattice coincide with the congruences of the $t$-lattice reduct and with  the congruences of the $k$-lattice reduct.  For the interlaced case, this result is obtained with the aid of the product representation and leads on  to a description of subdirectly irreducible algebras; see \cite{MPSV,RPhD,BJR11}.  
We exploited Proposition~\ref{lem:cong2}  to obtain   our $\ISP$ results 
for each of $\DB$, $\DBU$, $\DPB$ and $\DPBU$.
These results are of course immediate once the  subdirectly irreducible 
algebras are known, but our method of proof is much more direct.
Conversely, our results immediately yield descriptions of the subdirectly
irreducibles.

From
what is said above  
it might appear that, in certain aspects
 our approach leads to the same principal results as previous approaches do, albeit by a different route. But  we contend that  we have done much 
more than this. 
In our setting
we are able to harness  the techniques of 
natural duality  theory and to apply them in a systematic way 
to the best-known bilattice varieties.
We hereby  gain easy access to the applications  
presented, by way of illustration,  in Section~\ref{Sec:Applications}.
 It is true that the dualities developed in  \cite{MPSV} and in \cite{JR}
 can be described using our dualities and vice versa.  However 
 the deep connections
between 
 congruences of lattice reducts, our  $\ISP$ presentations, and the topological representation theory  only becomes clear using natural dualities.

We end  our paper with an interesting byproduct of our treatment 
which links back  
to the origins of bilattices.
The theory of bilattices and the investigation of four-valued logics have been intertwined ever since the concept of a bilattice was first introduced.  
In his seminal 
paper \cite{ND1} (also available  in \cite{ND2}),
Belnap introduced two lattices over the same four-element set $\{F,T,Both,None\}$,  the logical lattice $\alg{L4}$ and the approximation lattice $\alg{A4}$, the 
former admitting also a negation operation. 
With our notation, $\alg{L4}\cong(\{00,01,10,11\}; \lor_t,\land_t,\neg,00,11)$ and
  $\alg{A4}\cong \four_k$.
Belnap defines a \defn{set-up} as a map~$s$ from
 a set  $X$ of atomic formulas  
into $\{F,T,Both,None\}$, and extends  
$s$ in a unique way to a homomorphism
$\bar{s}\colon\alg{Fm}
(X)
\to \alg{L4}$,  
where 
$\alg{Fm}(X)$ to the set of formulas in the language $\{\land,\lor,\neg\}$.
 He then introduces a logic, understood as an entailment relation between formulas based on set-ups  
and what is nowadays called  
a Gentzen system which is
complete for this logic.
The connection between Belnap's logic and 
De Morgan lattices and  De Morgan algebras,
  hinted at in  
the definition of the former, was 
unveiled in detail by Font in \cite{F97} in the context of abstract algebraic logic.

 Belnap did more than just define his logic: he also presented a 
mathematical formulation 
of the  epistemic dynamic of the logic.
To do this, he  
defined \defn{epistemic states} as sets of set-ups and lifted the order 
on 
$\alg{A4}$ 
to  
a pre-order,  $\sqsubseteq$ (the  \defn{approximation order}),
between 
 epistemic states.
He  
then 
considered   
the partial order obtained from  
$\sqsubseteq$ by quotienting 
by the equivalence relation $\sqsubseteq\cap\sqsupseteq$ and  showed
 that the resulting poset  is isomorphic to the 
family of 
upward-closed sets 
of set-ups;  here set-ups are considered as elements of $\alg{A4}^{\alg{Fm}(X)}$ and are ordered pointwise. 
This emphasises the importance of the poset structure, 
as opposed to 
the algebraic structure, of $\alg{A4}$.
Furthermore, it is proved that, for each formula $A\in\alg{Fm}(X)$, the assignment $A\mapsto {\rm Tset}(A)=\{\, s\colon \bar{s}(A)\in\{T,Both\}\, \}$ maps conjunctions to intersections, disjunctions to unions and $\neg A\mapsto {\rm Tfalse}(A)=\{\,s\colon \bar{s}(A)\in\{F,Both\}\, \}$. 
So we could interpret Belnap's results as a representation of $\alg{Fm}(X)$ 
 as 
upward-closed 
subsets
of homomorphisms from $\alg{Fm}(X)$ to $\alg{L4}$  ordered pointwise 
by~$\alg{A4}$.

Only a few steps are needed to connect Belnap's representation, as outlined above,  
with the natural duality for De Morgan algebras; see \cite[Section~4.3.15]{CD98} and the references therein. 
We
adopt 
 the notation of  \cite{CD98} for the generating algebra, 
$\underline{\mathbf{dM}}$, and for the alter ego, $\twiddle{\mathbf{dM}}$. 
  First observe that
${\alg{L4}\cong\underline{\mathbf{dM}}}$ is a De Morgan algebra. Therefore
each homomorphism $h\colon\alg{Fm}(X)\to\alg{L4}$ factors 
 through the free De Morgan algebra $\mathbf{F}_{\cat{DM}}(X)$. Hence the set of set-ups can be identified with $\cat{DM}(\mathbf{F}_{\cat{DM}}(X),\alg{L4})$. 
It is also 
necessary to 
check that 
 for each formula $A$ the sets  $\text{Tset}(A)$ and $\text{Fset}(A)$ are  related 
by the involution of the dual space of a De Morgan algebra; more precisely, $g(\text{Tset}(A))=\text{Fset}(A)$.
 And finally, of course, 
topology plays its role by enabling one to characterise 
 those 
upward-closed sets (represented by maps) that correspond to formulas.

These  observations  
serve to stress that 
  $\alg{L4}$ and $\alg{A4
}$ in Belnap's works play quite different roles.
Moreover, these structures 
 are intimately related to the roles of 
$\underline{\mathbf{dM}}$ and $\twiddle{\mathbf{dM}}$
in the natural duality for De Morgan algebras. The idea of
combining 
 two lattices into one structure originated with 
 Ginsberg  \cite{Gin}. 
  The dualities
presented in Theorem~\ref{DBnatdual} and~~\ref{DBUnatdual} can be seen as a bridge reconciling Belnap's and Ginsberg's approaches, the first considering two separated lattice structures $\alg{L4}$ and $\alg{A4}$ with different roles but based on the same universe, and the latter combining them into a single 
algebraic structure.
We, in like manner,  
work with
 two different structures $\four$ and $\fourDBT$ (and $\fourU$ and $\fourDBTU$ in the unbounded case)  with  different structures
having distinctive  roles:  one logical with an algebraic structure, the other epistemic with a  poset structure.

\bibliographystyle{amsplain}
\renewcommand{\bibname}{References}

\end{document}